\documentclass{amsart}
\usepackage{amsmath,amsfonts,amssymb,amscd,amsthm,epsf}
\usepackage{pinlabel}
\usepackage{hyperref}
\usepackage{caption}
\newtheorem{prop}{Proposition}[section]
\newtheorem{thm}[prop]{Theorem}
\newtheorem{cor}[prop]{Corollary}
\newtheorem{sch}[prop]{Scholium}

\newtheorem{quess}[prop]{Questions}
\newtheorem{lem}[prop]{Lemma}

\newtheorem{adden}[prop]{Addendum}
\newtheorem{claim}[prop]{Claim}
\theoremstyle{definition}
\newtheorem{de}[prop]{Definition}
\newtheorem{example}[prop]{Example}

\newtheorem{ob}[prop]{Observation}
\theoremstyle{remark}
\newtheorem{Remark}[prop]{Remark}             
\newtheorem{Remarks}[prop]{Remarks}             

\def\Z{{\mathbb Z}}
\def\R{{\mathbb R}}
\def\Q{{\mathbb Q}}

\def\W{{\mathcal W}}
\def\D{{\mathcal D}}
\def\E{{\mathcal E}}
\def\F{{\mathcal F}}
\def\L{{\mathcal L}}
\def\p{{\mathfrak p}}
\def\pr{{\mathcal P}}
\def\inter{\mathop{\rm int}\nolimits}

\def\dim{\mathop{\rm dim}\nolimits}
\def\rk{\mathop{\rm rank}\nolimits}

\def\id{\mathop{\rm id}\nolimits}
\def\GL{\mathop{\rm GL}\nolimits}
\def\SL{\mathop{\rm SL}\nolimits}
\def\PD{\mathop{\rm PD}\nolimits}
\def\co{\colon\thinspace}
\def\SO{{\rm SO}}
\def\ds{{\Delta_{\rm spin}}}
\def\Arf{{\rm Arf}}

\begin{document}
\title{Transverse tori in Engel manifolds}
\author{Robert E. Gompf}
\thanks{This work was inspired by the April 2017 conference ``Engel Structures'' at the American Institute of Mathematics. The author would like to thank Roger Casals, Gordana Mati\'c, Emmy Murphy and the other participants for helpful discussions. He also thanks the referee for reading carefully, suggesting numerous expositional improvements and inspiring a few sharper results.}
\address{The University of Texas at Austin}
\email{gompf@math.utexas.edu}
\begin{abstract} We show that tori in Engel 4-manifolds behave analogously to knots in contact 3-manifolds: Every torus with trivial normal bundle is isotopic to infinitely many distinct transverse tori, distinguished locally (and globally in the nullhomologous case) by their formal invariants. (Few examples of transverse tori were previously known.) We classify the formal invariants, which are richer than for transverse knots. We show that in an overtwisted Engel structure, up to homotopy through such structures, these invariants are a complete set of uniqueness obstructions, and every torus with trivial normal bundle can be made transverse realizing any combination of these invariants. Fixing Engel structures not known to be overtwisted, we explore the range of the primary invariants of given tori. A sample application is that many Engel manifolds admit infinitely many transverse homotopy classes of {\em unknotted} transverse tori such that each class contains infinitely many transverse isotopy classes.

\end{abstract}
\maketitle


\section{Introduction}\label{Intro}

Engel structures, which only exist on 4-manifolds, are poorly understood cousins of contact structures in odd dimensions. The latter have been extensively studied in recent decades, first in dimension 3 and then in higher dimensions, and have been shown to be intimately related to the topology of the underlying manifolds. For example, contact topology was instrumental in proving the notorious Property~P Conjecture for 3-manifolds \cite{KM}. One of the main tools for studying contact 3-manifolds is the notion of a {\em transverse knot}. For example, these can be used to distinguish the crucial {\em tight} contact structures from the less interesting {\em overtwisted} structures. The most fundamental open problem in Engel topology is whether analogous tight Engel structures exist. If so, it seems reasonable to expect that they should become a powerful tool for studying 4-manifolds. With this background in 2017, an AIM conference on Engel structures was held, at which Eliashberg asked, in light of transverse knots in contact 3-manifolds, what could be said about transverse tori in Engel manifolds. The present article shows that transverse tori behave analogously to transverse knots in contact 3-manifolds, although the Engel case is richer in some aspects than the contact case.

\subsection{Background} Engel structures naturally emerge from Cartan's study of $k$-plane fields in $n$-manifolds \cite{C}. A modern exposition appears in \cite{P}; we review the most relevant parts in Section~\ref{Background}. Cartan classified all {\em topologically stable} subsets of the space of such distributions. By definition, these are open subsets consisting of distributions without local invariants. That is, they are all described by the same local model at every point. By a dimension count, there cannot be enough diffeomorphisms locally to generate an open set in the space of distributions unless $k$ or $n-k$ is small. The most obvious case, $k=1$, is the line fields on any manifold. When $n-k=1$, the topologically stable hyperplane fields are the contact structures ($n$ odd) and even-contact structures ($n$ even). The only remaining case, $k=n-k=2$, is the Engel structures. These structures are all characterized as being ``maximally nonintegrable'' in a sense discussed in Section~\ref{Background}.

Much of the power of contact topology comes from the tight/overtwisted dichotomy. Overtwisted contact structures satisfy the {\em h-Principle} ({\em Homotopy Principle}) of Eliashberg and Gromov \cite{EM}, \cite{Gr}. That is, there is a unique homotopy class of overtwisted contact structures within each homotopy class of hyperplane fields endowed with suitable auxiliary data (that is vacuous for oriented 3-manifolds) \cite{E}, \cite{BEM}. Thus, their classification is a problem in algebraic topology. In contrast, tight contact structures appear sporadically, depending in a delicate way on the topology of the underlying manifold. Even-contact structures are considered less interesting since McDuff showed they all obey the h-Principle \cite{McD}. That is, there are no tight even-contact structures. There are recently developed notions of overtwisted \cite{PV} and {\em loose} \cite{CPP} Engel structures, both of which satisfy the h-Principle. However, the relation between these notions is not yet known, and the fundamental question remains of whether there are additional homotopy classes of ``tight'' Engel structures.

Tight contact structures on 3-manifolds are detected by which {\em transverse knots} they contain. These are defined to be circles embedded transversely to the contact planes. Immersed transverse circles satisfy the h-Principle, with each homotopy class containing a unique transverse knot up to homotopy through transverse immersions. However, the h-Principle fails for (embedded) transverse knots, due to a ``formal'' invariant determined by the underlying bundle theory. In fact, tight contact structures are characterized by the failure of nullhomologous knots to transversely realize large values of this invariant (and realization need not be unique). To similarly understand Engel structures, it is natural to look for closed surfaces transverse to an Engel plane field. It is not hard to see (Section~\ref{TransSurf}) that any such surface must be a torus with trivial normal bundle. (We assume throughout the paper that everything is orientable; otherwise transverse Klein bottles can exist, Remark~\ref{Klein}(a).) Transverse {\em immersions} are again classified by the h-Principle \cite[Theorem~31]{PP} (Theorem~4 in the arXiv version). However, unlike circles in 3-manifolds, immersed surfaces in 4-manifolds cannot usually be perturbed to embeddings, making the existence question for transverse embeddings more difficult. In fact, only a few examples of transverse tori were previously known \cite{PV}.

\subsection{Results} Our first theorem, proved in Section~\ref{MainProof}, completely solves the existence problem in 4-manifolds with (oriented) Engel structures.

\begin{thm}\label{main0}
A closed surface embedded in an Engel manifold is isotopic to a transverse surface if and only if it is a torus with trivial normal bundle. If so, the isotopy can be assumed $C^0$-small.
\end{thm}

\noindent Our main technique is to reduce to contact topology, which inspires our methods in multiple ways. For example, the notion of transverse pushoffs of Legendrian knots adapts to tori (Section~\ref{Surface}, Proposition~\ref{push}).

This theorem raises the dual question of enumerating transverse representatives of an isotopy class, up to isotopy through transverse embeddings, i.e., {\em transverse isotopy}. For comparison, in a contact plane field with vanishing Euler number, nullhomologous transverse knots can be distinguished by their {\em self-linking number} in $\Z$, their unique formal transverse isotopy invariant. (More subtle invariants are beyond the scope of this paper.) Every such knot is $C^0$-small isotopic to infinitely many transverse knots, distinguished by their self-linking numbers. For homologically essential knots, the same applies in a tubular neighborhood, but not globally in general. For nullhomologous knots with overtwisted complements, the self-linking number classifies transverse representatives and takes all odd values. But in the tight case, both of these statements fail.

We analyze the formal invariants of transverse tori in Section~\ref{Distinguishing} and apply them in Section~\ref{Examples}. These invariants behave analogously to the self-linking number in many ways, but are also richer. A version of the invariant directly analogous to the self-linking number has been used by Kegel~\cite{K} to distinguish certain nullhomologous transverse tori in Engel manifolds (see Example~\ref{PV}(a)). However, transverse tori also have other formal invariants. We classify these in Section~\ref{Classification} (Theorem~\ref{formal}). Informally, we have: 

\begin{claim} The formal invariants of a transverse torus $\Sigma$ consist of a pair of classes in $H^1(\Sigma)$ and a pair of secondary invariants in $\Z$.
\end{claim}

 Using the h-Principle for overtwisted Engel structures \cite{PV}, we show in Section~\ref{Overtwisted} that in the overtwisted setting, transverse tori are flexible in a similar sense to transverse knots. Unlike for contact structures, however, a compactly supported homotopy through Engel structures need not be realized by an isotopy of the manifold. Thus, to realize this flexibility, we must allow the Engel structure to vary by homotopy. We informally summarize here the results given more precisely in that section:

\begin{claim} In overtwisted Engel structures up to homotopy: Every torus with trivial normal bundle can be made transverse, and then modified by any combination of the formal invariants. An isotopy between fixed transverse tori can be assumed to be a transverse isotopy if and only if its formal invariants vanish.
\end{claim}

In contrast, the formal invariants seem less flexible when the Engel structure is held fixed. This seems especially interesting in Engel structures not known to be overtwisted (i.e. potentially tight structures). We analyze the primary invariants in the last two sections, beginning with the analogue $\Delta_\nu\in H^1(\Sigma)$ of the self-linking number in Section~\ref{Dnu}. One application is a general nonuniqueness theorem reminiscent of transverse knots:

\begin{thm}\label{many}
Let $\Sigma$ be a torus with trivial normal bundle in an Engel manifold $M$. Then
\item[a)] After a $C^0$-small isotopy, there is a neighborhood $U$ in which $\Sigma$ is $C^0$-small isotopic to infinitely many transverse tori, no two of which are transversely isotopic in $U$.
\item[b)] If $[\Sigma]$ vanishes in $H_2(M)$ then no two of these tori are transversely isotopic in $M$.
\end{thm}

With embedded surfaces, we encounter a subtlety that does not arise for oriented circles: There may be an isotopy that sends $\Sigma$ onto itself in a way that is nontrivial on $H^1(\Sigma)$. This reparametrizes $\Sigma$ so that $\Delta_\nu$ may appear different, even though the image surface is unchanged. (Such behavior already occurs for unknotted tori, Example~\ref{PV}(d).) In the above theorem, and elsewhere unless otherwise specified, we mean that the surfaces are not transversely isotopic for any choices of parametrization. This follows by using the divisibility of $\Delta_\nu$, which is preserved by automorphisms of $H^1(\Sigma)$. In spite of this ambiguity, the difference $D_\nu(F)$ of two values of $\Delta_\nu$ is still well-defined and useful for a fixed isotopy $F$, and vanishes when the isotopy is transverse. In this sense, $\Delta_\nu$ is a transverse isotopy invariant.

By construction, the tori arising from a given $\Sigma$ in the above theorem are all {\em transversely homotopic}, that is, homotopic through immersed transverse surfaces. In particular, $\Delta_\nu$ is not a transverse homotopy invariant, nor is its divisibility. This reflects the corresponding failure of the self-linking number of transverse knots in contact 3-manifolds, for which homotopy implies transverse homotopy. In contrast, the other primary invariant $\Delta_T\in H^1(\Sigma)$ {\em is} a transverse homotopy invariant. Unlike $\Delta_\nu$, this is well-defined even for homologically essential transverse tori, and it has no analogue for transverse knots. We study the range of this invariant in Section~\ref{DT}. We find that even unknotted transverse tori (isotopic to $S^1\times S^1$ in some $\R^2\times\R^2$ chart) need not be transversely homotopic:

\begin{thm}\label{nonhom}
Every circle bundle with even Euler class over a 3-manifold has a compatible Engel structure admitting infinitely many transverse homotopy classes of unknotted transverse tori such that each class contains infinitely many transverse isotopy classes.
\end{thm}

\noindent These Engel manifolds arise by {\em prolongation} (Section~\ref{Engel}) of (typically overtwisted) contact structures chosen from any homotopy class of plane fields on any 3-manifold. The theorem follows from Theorem~\ref{prolong}, which provides a plethora of knot types of tori in prolongations satisfying the analogous conclusion.

\begin{Remark}\label{C0small}
While Theorems~\ref{main0} and \ref{many} are analogous to the behavior of knots in contact 3-manifolds, Theorem~\ref{nonhom} has no such analogue since homotopy implies transverse homotopy for transverse knots. We can alternatively compare these theorems with the more restricted behavior of surfaces in contact 3-manifolds. In that setting, important moves such as the bypass operation (Section~\ref{BypassCont}) cannot be made $C^0$-small. That might suggest (Remark~\ref{notC0}) simplifying the current paper by dropping the $C^0$-small conclusions and just working in a small neighborhood of $\Sigma$ (although we still seem to need the natural $C^0$-small conclusion in Lemma~\ref{zigzag} arising from making knots Legendrian, Proposition~\ref{helices}). However, the analogy with knots makes the stronger statements seem most natural and worth the few extra technicalities. As an example of the extra freedom arising from the fourth dimension, Theorem~\ref{main0} relies on an Engel version of the bypass move (Section~\ref{BypassSub}) that is both $C^0$-small and more flexible than the contact version.
\end{Remark}

\subsection{Future directions}\label{Future} The tools developed in this paper have potential application to the problem of recognizing tight Engel structures (if any exist). This is more subtle than the corresponding problem for closed contact manifolds, since it is not clear whether overtwistedness is preserved by homotopy through Engel structures. Thus, while each component $\mathfrak E$ of the space of formal Engel structures on $M$ contains a unique component of overtwisted Engel structures \cite{PV}, this may lie in a strictly larger component of the subset of all Engel structures in $\mathfrak E$. We would ideally like to find other components of the latter, but a preliminary step would be to recognize any Engel structure in $\mathfrak E$ that is not overtwisted. In the contact setting, this can be done via restrictions on the self-linking of transverse knots, suggesting an analogous approach for Engel structures. While this paper realizes many pairs $(\Delta_T,\Delta_\nu)$ for fixed isotopy classes of tori in fixed Engel structures, various gaps remain. In addition, we have no concrete examples exploiting the secondary formal invariants. For an Engel structure in $\mathfrak E$, one could hope to find a torus that cannot be made transverse with certain values of the formal invariants, while these values are realized in all overtwisted Engel structures in $\mathfrak E$. (We show that each value is realized by {\em some} such overtwisted structure in Section~\ref{Overtwisted}.) Observation~\ref{tight} gives a natural nested family $M_r$, $0<r\le\infty$, of Engel manifolds diffeomorphic to $\R^4$ such that every Engel manifold contains all bounded regions of $M_r$ for all sufficiently small $r$. Thus, if every $M_r$ is overtwisted then all Engel manifolds are, making these the most likely candidates for tight Engel manifolds. This suggests the utility of studying transverse torus theory in $\R^4$, analogously to transverse knot theory in $\R^3$. The gaps in results of this paper suggest (for example) the following questions:

\begin{quess}\label{ques}
a) In $M_r$, can an unknotted torus be transverse with $\Delta_T\ne0$? What about knotted tori? Do the answers depend on $r$?

\item[b)] In a fixed Engel manifold, does every family of isotopic tori with well-defined values of $\Delta_\nu\in H^1(\Sigma)$ (Section~\ref{DeltaNu}) have a class $\delta$ for which each pair $(\Delta_T,\Delta_\nu+\delta)$ is linearly dependent (cf.~Example~\ref{vert})?

\item[c)] Is there a pair of transverse tori in $M_r$ (or any Engel manifold) that are not transversely isotopic but are related by an isotopy $F$ with vanishing formal invariants (or just with the same $\Delta_T$, and $D_\nu(F)=0$)?
\end{quess}

\subsection{Organization and conventions} After reviewing the additional necessary background on contact and Engel topology in Section~\ref{Background}, we prove Theorem~\ref{main0} in Section~\ref{Main}, making a given torus transverse by isotopy. Section~\ref{Distinguishing} defines and classifies the formal invariants of transverse tori. Finally, Section~\ref{Examples} explores the range of the invariants, proves Theorems~\ref{many} and \ref{nonhom}, and explicitly computes the primary invariants in some examples. Sections ~\ref{Distinguishing} and \ref{Examples} can be read independently of Section~\ref{Main}, and proofs of the latter two theorems and the analysis of the primary invariants (Sections~\ref{Dnu} and \ref{DT}) can be read without Sections~\ref{Classification}, \ref{Mod2} and \ref{Overtwisted}.

We use the following conventions throughout the paper, except where otherwise indicated: Homology and cohomology have integral coefficients, with $\PD$ denoting Poincar\'e duality or its inverse. We work in the category of smooth, connected manifolds. Curves and surfaces are assumed to be closed (and the latter are often tori). Other manifolds are allowed to be noncompact but (for simplicity) without boundary. Manifolds and distributions on them are assumed to be oriented, compatibly. The exact meaning of compatibility is only needed for some signs in Sections~\ref{Dnu} and \ref{DT}. However, for future research it seems important to specify the meaning in a way that is optimally compatible with the standard conventions of smooth and contact topology, so we discuss the details carefully in Section~\ref{Orientation}.


\section{Further background}\label{Background}

This section reviews various ways of visualizing and manipulating contact 3-manifolds, Engel 4-manifolds and their submanifolds, as well as presenting Lemma~\ref{zigzag} and Addendum~\ref{C0} for later use, and establishing natural orientation conventions for Engel topology. To begin, we must understand the meaning of ``maximal nonintegrability'' that characterizes topologically stable distributions. We consider hyperplane fields here and Engel 2-plane fields in Section~\ref{Engel}. For the former, maximal nonintegrability means that the Lie bracket operation $[u,v]$ on vector fields is maximally nondegenerate on the hyperplane field:  Our orientability hypotheses guarantee that every hyperplane field in a manifold is the kernel of some 1-form $\alpha$, unique up to scale. The standard formula
$$d\alpha(u,v)=u\alpha(v)-v\alpha(u)-\alpha[u,v]$$
implies that when $u$ and $v$ are vector fields in $\ker\alpha$, $d\alpha(u,v)=-\alpha[u,v]=\alpha[v,u]$. This interprets the normal component of the Lie bracket on $\ker\alpha$ as a pointwise, bilinear form. Maximal nonintegrability of a contact or even-contact structure is then maximal nondegeneracy of $d\alpha|\ker\alpha$. For contact structures, $\ker\alpha$ has even dimension, so this means $d\alpha|\ker\alpha$ is symplectic. On a 3-manifold $N$, this just says $d\alpha$ is never 0 on the 2-planes $\xi$, so it is an area form on them that we always assume is positive. Equivalently, $\alpha\wedge d\alpha$ is a positive volume form on $N$. For an even-contact structure $\E$,  the hyperplanes have odd dimension, so maximal nondegeneracy means there is a canonical line field $\W$ in $\E$ such that $d\alpha$ descends to a well-defined symplectic form on $\E/\W$. On a 4-manifold $M$, this occurs whenever $d\alpha|\E$ is never 0. In terms of Lie brackets, nonintegrability on 3- and 4-manifolds is given by the conditions $[\xi,\xi]=TN$ and $[\E,\E]=TM$, respectively. For even-contact structures of any dimension, $\W$ is equivalently characterized by the condition $[\W,\E]\subset\E$. Thus, the flow of any vector field in $\W$ preserves $\E$, since its Lie derivative preserves the set of vector fields in $\E$. On any open subset of $M$ whose quotient by such a flow is a manifold, the latter then canonically inherits a contact structure $\xi=\E/\W$. Similarly, any hypersurface $N\subset M$ transverse to $\W$ inherits a contact structure $\xi=\E\cap TN$ that is invariant under such flows and locally projects to the canonical contact structure $\E/\W$.

\subsection{Contact topology}\label{Contact}

We now review contact 3-manifolds and their submanifolds, deferring the fundamental topic of convex surfaces to Section~\ref{Convex}. See, e.g., \cite{OS} for more details.

\subsubsection{Knots and their formal invariants}\label{Knots} First we consider knots suitably compatible with the ambient (oriented) contact plane field $\xi=\ker\alpha$. {\em Transverse} knots are everywhere transverse to $\xi$, whereas {\em Legendrian} knots are everywhere tangent to $\xi$. If $K$ is transverse, it is canonically oriented by the condition $\alpha|K>0$, whereas a Legendrian $K$ has $\alpha|K=0$ everywhere, so can be oriented arbitrarily. Transverse knots have a unique formal invariant. This arises from a relative invariant associated to each regular homotopy between transverse knots (namely, the difference between relative Euler classes of $\xi|K$ and the normal bundle $\nu K$ pulled back over the domain $I\times S^1$). When $K$ is nullhomologous and the Euler class $e(\xi)$ vanishes, it becomes an absolute invariant, the {\em self-linking number} $l(K)\in\Z$. (When $e(\xi)\ne0$, this is still defined relative to a preassigned Seifert surface $\Sigma$ since $e(\xi|\Sigma)=0$.) The self-linking number is defined to be the winding number along $K$ of any nowhere-zero section of $\xi$ on $N$, relative to the 0-framing (which is a vector field outward normal to any Seifert surface, that we can assume lies in $\xi$ since $K$ is transverse). Consideration of spin structures shows that this is always odd. A similar discussion of Legendrian knots yields two formal invariants $tb(K), r(K)\in\Z$. We only need $tb(K)$, which measures a vector field transverse to $\xi$ along $K$ relative to the 0-framing. We can now characterize tight contact structures in four ways: There is no transverse unknot with $l(K)\ge0$ or Legendrian unknot with $tb(K)\ge0$, and every nullhomologous knot has an upper bound on $tb$ of Legendrian representatives, or (if $e(\xi)=0$ or for a fixed Seifert surface) on $l$ of transverse representatives.

\subsubsection{Local models}\label{Models} To construct local models of subsets of contact manifolds, consider the standard contact structure on $\R^3$, which we usually describe as the kernel of $\alpha=dz+xdy$. This is the unique tight contact structure on $\R^3$ up to contactomorphism (diffeomorphism preserving the contact plane field), although we sometimes describe it using other contactomorphic plane fields. Uniqueness guarantees that every point of a contact 3-manifold has a neighborhood contactomorphic to the standard $\R^3$. We will introduce various other local models as needed. For example, every Legendrian knot in a contact 3-manifold has a neighborhood pairwise contactomorphic to a neighborhood of the $y$-axis in $(\R^3,dz+xdy)$ mod unit $y$-translation. Similarly, every transverse knot has a neighborhood contactomorphic to a neighborhood of the $z$-axis mod unit $z$-translation. In the latter case, it is a bit more natural to use the cylindrically symmetric contact form given by $\alpha'=dz+\frac12(xdy-ydx)=dz+\frac12 r^2d\theta$ in cylindrical coordinates. This is related to $\alpha$ by the contactomorphism $\varphi(x,y,z)=(x,y, z+\frac12 xy)$ with $\varphi^*(\alpha')=\alpha$.

\subsubsection{Projections of knots}\label{Proj}
Knots in $(\R^3,dz+xdy)$ can be described by projections. The {\em front projection} $(x,y,z)\mapsto (y,z)$ sends each contact plane to a line, whose slope $\frac{dz}{dy}=-x$ recovers the deleted coordinate. These lines realize all nonvertical slopes and are co-oriented upward. We can now represent a knot by its image in the projection, together with lines whose slopes recover $x$ as in Figure~\ref{stab}. Then a transverse knot projects to an immersion that is everywhere positively transverse to the lines as in the figure. (Note that vertical tangents must be oriented upward, and a positive crossing, drawn in standard ``x'' position, must have at least one upward strand.) The self-linking number is then the winding number of the {\em blackboard framing} given by a vector field in $\xi$ parallel to the $x$-axis, relative to the 0-framing; this is just the signed number of crossings in the diagram. A generic knot has finitely many tangencies to $\xi$ that become tangencies to lines in the projection. The image of a Legendrian knot is everywhere tangent to the lines, which then need not be drawn. Thus, we recover the $x$-coordinate from its slope everywhere. Since the slope cannot be vertical, the projection must have cusps at which the knot is parallel to the $x$-axis. The formal invariants can be read from the diagram by suitably counting crossings and cusps. A Legendrian knot can also be described by its {\em Lagrangian projection} $(x,y,z)\mapsto (x,y)$. We recover the $z$-coordinate up to a constant as $\Delta z=-\int xdy$. The integral can be interpreted as a signed area by Green's Theorem. This must vanish when we traverse the entire image of a closed Legendrian curve in $\R^3$. Similarly, two Legendrian arcs $C$ and $C'$ with the same initial point will have the same endpoint if and only if these endpoints have the same Lagrangian projection and the enclosed signed area vanishes.

\begin{figure}
\labellist
\small\hair 2pt
\pinlabel {$\xi$} at 25 25
\pinlabel {$\xi$} at 183 14
\endlabellist
\centering
\includegraphics{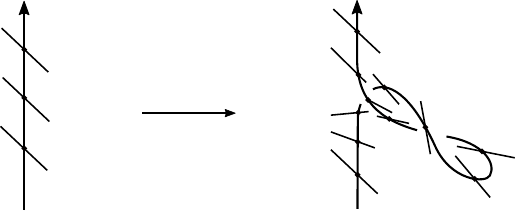}
\caption{Stabilizing a transverse knot by a $C^0$-small isotopy.}
\label{stab}
\end{figure}

\subsubsection{Operations on knots}\label{KnotOps} There are several natural operations on Legendrian and transverse knots. Both kinds of knots can be {\em stabilized}, lowering $tb$ by 1 or $l$ by 2, respectively. Figure~\ref{stab} shows the front projection of this procedure for a transverse knot in the standard $\R^3$, and is a local model for the general case. The altered curve can be kept transverse by suitably controlling the $x$-coordinate (as measured by $\xi$ in the figure). By keeping the new loops narrow and changes in $x$ small, we can arrange the stabilization to result from either an arbitrarily $C^0$-small isotopy, or a $C^0$-small transverse homotopy. The homotopy carries along the blackboard framing on $\xi$, but the isotopy twists the Seifert surface, so the self-linking drops by 2 (as also seen by counting crossings with sign). 

Another natural operation is the {\em transverse pushoff} changing an oriented Legendrian knot $K$ to a transverse knot $\tau K$. More precisely, there is an embedding of $\R\times S^1$ such that $\{t\}\times S^1$ maps onto $K$ when $t=0$ and is transverse otherwise, with orientation depending on the sign of $t$ (so each orientation on $K$ is realized as a transverse pushoff). The annulus is apparent in the image of the $xy$-plane in the local model of $K$ from Section~\ref{Models}. Alternatively, this operation can be derived using the contact condition and Lie derivative. We will adapt the latter method to tori in Engel manifolds as Proposition~\ref{push}.

There is also a standard (but not canonical) procedure for making an arbitrary knot Legendrian. We give a proof as preparation for a necessary lemma:

\begin{prop}\label{helices}\cite{E}.
Every knot $K$ in a contact 3-manifold is $C^0$-small isotopic to a Legendrian knot, and hence, to a transverse knot with either orientation.
\end{prop}

\begin{proof}
If $K$ is transverse, it is locally modeled by the $z$-axis in $(\R^3,dz+\frac12 r^2d\theta)$ mod unit $z$-translation. The boundary of an $\epsilon$-neighborhood of this is foliated by Legendrian helices, whose slope becomes arbitrarily small as $\epsilon$ decreases. By choosing $\epsilon$ carefully, we can arrange each helix in $\R^3$ to descend to a closed curve in the model that is isotopic to the image of the $z$-axis, as required. If $K$ is not transverse, it generically has a nonempty, finite set of tangencies to the contact planes. Perturb $K$ to be piecewise smooth, Legendrian near these tangencies, and transverse (with alternating orientations) on the smooth segments in between. The previous local model replaces these segments by Legendrian helices. We can smooth the resulting Legendrian knot locally using a front projection near each corner.
\end{proof}

For a generic knot $K$ in $(\R^3,dz+xdy)$, its front projection has finitely many tangencies to the image lines of $\xi$. The above procedure adjusts $x$ to make $K$ Legendrian near these. It then replaces each remaining segment in the projection by a {\em zig-zag arc}, a cusped arc $C^0$-close to the original, that repeatedly crosses it to create alternating {\em teeth} (cusped triangular regions bounded by the new and old arcs). The $C^0$-small conclusion arranges its tangent lines to be arbitrarily close to the lines recording the $x$-coordinate of $K$. Such a zig-zag arc can be assumed 1:1 when the original segment is. (This structure is all visible for the standard model transverse arc, viewed in $(\R^3, dz+xdy)$. We can shear its front projection to realize any constant slopes for $\xi$ by $(x,y,z)\mapsto(x-m,y,z+my)$. The front projection of any transverse arc can then be locally approximated by such a model after rescaling the arc by $(x,y,z)\mapsto(x,cy,cz)$ for large $c$.)

\subsubsection{Surfaces in contact 3-manifolds}\label{SurfCont} Surfaces necessarily interact with contact structures in more complicated ways than curves. For any surface $\Sigma$, the subspaces $\xi\cap T\Sigma$ have dimension at least 1 everywhere. On the open subset of $\Sigma$ where the dimension is 1, we have a {\em characteristic line field}, oriented via the orientations on $\xi$ and $\Sigma$, and integrating to the {\em characteristic foliation}. This is singular at the other points of $\Sigma$. The singular set cannot contain an open subset, since $\xi$ would be closed under Lie bracket there. In fact, the contact condition implies the singularities have nonzero divergence in the characteristic foliation, so the singular set lies in a 1-manifold. Generically, it consists of isolated points, but the resulting singular foliation may still be quite complicated. Fortunately, perturbing $\Sigma$ allows us substantial control of the characteristic foliation, as discussed in the next section. This foliation, in turn, determines the contact structure in a neighborhood of $\Sigma$ when, for example, $\partial\Sigma$ is empty or Legendrian. We will need one other simple example:

\begin{prop}\label{annuli}
Let $C$ denote $I$ or $S^1$, and for $i=0,1$, let $f_i\co I\times C\to N_i$ be embeddings of the square or annulus into 3-manifolds with contact forms $\alpha_i$. Suppose the (oriented) characteristic foliation of each is the image of the product foliation $I\times \{x\}$. Then $f_1\circ f_0^{-1}$ extends to a contactomorphism of neighborhoods of the images.
\end{prop}

\begin{proof}
(Sketch.) The map $f_1\circ f_0^{-1}$ extends to a diffeomorphism $\varphi$ of neighborhoods. By uniqueness of tubular neighborhoods, we can assume $\varphi$ preserves the (oriented) contact planes at all points on the image $A$ of $f_0$, so after rescaling $\alpha_1$, we have $\varphi^*\alpha_1=\alpha_0$ on $TN_0|A$. A standard argument (Moser's method) isotopes $\varphi$ to the required contactomorphism near $A$, by integrating a time-dependent vector field constructed from $\varphi^*\alpha_1-\alpha_0$. Since this vanishes on $TN_0|A$, $A$ is fixed.
\end{proof}

\subsubsection{A key lemma} We can now present our final lemma on making arcs Legendrian. The manifold $S^1\times\R^2$ with tight contact form $
\sin(w)dx+\cos(w)dy$ has a natural front projection $(w,x,y)\mapsto(x,y)$ analogous to that of $\R^3$. If an arc $C$ in $S^1\times\R^2$ projects diffeomorphically (up to orientation) to an interval on the $x$-axis, its $w$-coordinate must everywhere be 0 or $\pi$ if it is Legendrian, and in an open semicircle between these values if it is transverse.

\begin{lem}\label{zigzag}
Consider a smooth family $C_s$ of arcs as above, with $s$ ranging over a compact parameter manifold. Suppose each arc projects as above to $[-1,2]$ on the $x$-axis and is transverse over $(0,1)$ and Legendrian elsewhere, with the family independent of $s$ near $x=0,1$. Then there is a smooth family of $C^0$-small isotopies supported over $[0,1]$ making the arcs Legendrian, with front projections given over $[0,1]$ by embedded zig-zag arcs (as described after Proposition~\ref{helices}) whose intersections with the $x$-axis are independent of $s$. These can be arranged so that for each $s$, the pair of teeth adjacent to each positive intersection have equal area (of opposite sign).
\end{lem}

\begin{proof}
Parametrize each arc $C_s$ by $x$. Then the arcs all agree outside some compact subset of $(0,1)$; let $I^x\subset(0,1)$ be a compact interval whose interior contains this. There is a compact interval $I^w$ in $S^1-\{0,\pi\}$ such that the rectangle $A=I^w\times I^x\times\{0\}\subset S^1\times\R^2$ contains each $C_s|\inter I^x$ in its interior. Then $A$ has a nonsingular characteristic foliation whose leaves have constant $x$. For each $s$, there is a self-diffeomorphism of $A$ preserving the $x$-coordinate, sending $C_s|I^x$ to a preassigned $C_{s_0}|I^x$, and with support projecting into $\inter I^x$. The previous proposition extends this to a contactomorphism of neighborhoods that (by the proof of the proposition) extends by the identity for $x\notin I^x$ to a neighborhood of $C_s$. For an interval $J\subset(0,1)$ containing $I^x$ in its interior, a similar contactomorphism sends a neighborhood $U$ of $C_{s_0}|J$ to the model transverse arc in $(\R^3,dz+xdy)$, with $A\cap U$ mapping into the plane $y=0$. The standard helix of any sufficiently small radius $\epsilon$ then pulls back into $U$. Smoothly joining this helix to the Legendrian parts of $C_{s_0}$ (as in the proof of Proposition~\ref{helices}, working over the small intervals $[0,1]-J$) gives a Legendrian arc isotopic to $C_{s_0}$.  Pulling back to each $C_s$ gives the required family of $C^0$-small isotopies to Legendrian arcs (since all steps in the construction are smooth in $s$). The image zig-zag arcs all have the same intersections with the $x$-axis (projected from intersections with $A$). Choosing $\epsilon$ small enough allows the flexibility to control areas as required, by suitably extending the smaller teeth.
\end{proof}

\subsection{Convex surfaces in contact topology}\label{Convex}

Next we sketch what we need from convex surface theory, which was pioneered by Giroux~\cite{G} and subsequently Honda~\cite{H}. (See the latter for a broader discussion.) We also present an addendum for later use. A {\em contact vector field} in a contact manifold is one whose flow preserves the contact structure. A surface is {\em convex} if it is transverse to such a vector field. Every closed surface in a contact manifold is $C^\infty$-small isotopic to a convex surface. The same holds for a compact surface with connected, Legendrian boundary $C$, after a $C^0$-small perturbation rel $C$ near $C$, provided $tb(C)\le 0$. (If there are more boundary components, the same holds when each satisfies an analogous condition.) If $\Sigma$ is convex with respect to a contact vector field $v$, its {\em dividing set} $\Gamma$ is the subset of $\Sigma$ on which $v$ lies in $\xi$. This is always a compact 1-manifold, with boundary in $\partial\Sigma$, whose components are called {\em dividing curves}. The dividing set is independent of choice of $v$, up to isotopy in $\Sigma$ preserving the characteristic foliation. It is nonempty, and splits $\Sigma$ into two (often disconnected) regions $R_+$ and $R_-$, defined by whether projecting out $v$ preserves or reverses orientation as an isomorphism from $\xi_p$ to $T\Sigma_p$. The characteristic foliation is always transverse to $\Gamma$, directed from $R_+$ to $R_-$. Clearly, for a closed surface $\Sigma$, the Euler characteristics satisfy $\chi(R_+)+\chi(R_-)=\chi(\Sigma)=\langle e(T\Sigma),\Sigma\rangle$. The orientation-reversing bundle map over $R_-$ also gives $\chi(R_+)-\chi(R_-)=\langle e(\xi),\Sigma\rangle$. For example, $\chi(R_\pm)=0$ for a torus with $\xi|\Sigma$ trivial. The {\em Giroux Criterion} says that a convex surface other than $S^2$ (with Legendrian boundary allowed) has a tight neighborhood if and only if $\Gamma$ has no inessential circles. (For $S^2$, the corresponding condition is that $\Gamma$ should be a single circle.) Thus, for a convex torus with a tight neighborhood, $\Gamma$ must consist of a nonzero, even number of parallel essential circles, so $\langle e(\xi),\Sigma\rangle=0$. (Constraints for surfaces of other genera follow similarly.)

\subsubsection{Giroux Flexibility and a useful addendum} From the viewpoint of contact topology, the dividing set of a convex surface captures the essential information of its characteristic foliation. (We will see in Examples~\ref{liftReeb} and \ref{transLift} that it is less complete in the Engel setting.) An oriented singular foliation on a compact surface $\Sigma$ with boundary is {\em adapted} to a co-oriented 1-manifold $\Gamma$ embedded rel boundary in $\Sigma$ if $\Sigma$ embeds as a convex surface with Legendrian boundary in some contact 3-manifold, realizing the given foliation and dividing set. This is a very weak condition; we give our main examples below. For a fixed $\Gamma$, such foliations are interchangeable:

\begin{thm}[Giroux Flexibility Theorem~\cite{G}]\label{flex}
In a contact 3-manifold $(N,\xi)$, let $\Sigma$ be a compact surface with Legendrian boundary (possibly empty), convex with respect to a contact vector field $v$ and with dividing set $\Gamma$. Let $\F$ be another singular foliation adapted to $\Gamma$. Then there is an isotopy $\varphi_t\co\Sigma\to N$, $t\in[0,1]$, with $\varphi_0$ the inclusion, fixing $\Gamma$ and preserving $\partial \Sigma$, such that $\varphi_1$ maps $\F$ to the characteristic foliation on $\varphi_1(\Sigma)$ and each $\varphi_t(\Sigma)$ is transverse to $v$. \qed
\end{thm}

\noindent The isotopy constructed here is not $C^0$-small. For example, if $\F$ has a unique singularity and is obtained from the characteristic foliation of $\Sigma$ by an isotopy of $\Sigma$, we should expect the singular point to move a large distance. However, this can be adequately remedied:

\begin{adden}\label{C0}
After an isotopy of $\F$ rel $\Gamma$ in $\Sigma$, the isotopy $\varphi_t$ can be assumed to be $C^0$-small.
\end{adden}

\begin{Remark}\label{notC0}
We could avoid this addendum and several other technicalities by dropping the $C^0$ conclusions from our theorems. However, the stronger conclusions seem more natural (Remark~\ref{C0small}). While the addendum is new (to the author's knowledge), the proof uses standard techniques, and the addendum probably would already be a standard tool in contact 3-manifold topology if more surface operations could be made small (cf. Section~\ref{BypassCont}).
\end{Remark}

\begin{proof}[Proof of Addendum~\ref{C0}]
The flow of the contact vector field $v$ maps $\Sigma\times\R$ into $N$, identifying $\Sigma\times\{0\}$ with $\Sigma$ and sending the unit vector field along $\R$ to $v$, so that the pullback of $\xi$ is given by an $\R$-invariant contact form $\alpha_1$. Giroux actually constructs his isotopy $\varphi_t$ in this product. By compactness, the image of the isotopy projects into an interval of the form $[-a,a]$ in $\R$. For sufficiently small $\epsilon>0$, we wish to find an isotopy of the form $\psi_s\times\eta_s\co\Sigma\times\R\to\Sigma\times\R$, $s\in[\epsilon,1]$, through contactomorphisms, with $\psi_1\times\eta_1$ the identity, $\eta_s(z)=sz$ and $\psi_s$ fixing $\Gamma$. Then $\psi_\epsilon\times\eta_\epsilon$ contactomorphically squeezes $\Sigma\times[-a,a]$ onto $\Sigma\times[-\epsilon a,\epsilon a]$, preserving the product structure. Thus, the conclusion of the Flexibility Theorem is still satisfied if we replace $\varphi_t$ by $\varphi'_t=(\psi_\epsilon\times\eta_\epsilon)\circ\varphi_t\circ \psi_\epsilon^{-1}$ and $\F$ by the isotopic foliation $\F'=\psi_\epsilon(\F)$. Since $\varphi'_1(\Sigma)$ is transverse to $v$, we can interpret it as the graph of a function $f\co\Sigma\to\R$ that can be assumed arbitrarily $C^0$-small by choice of $\epsilon$. The theorem and addendum are now satisfied by setting $\varphi''_t=\id_\Sigma\times (tf)$ and using the foliation $\pi\circ\varphi'_1(\F')$, where $\pi$ is projection to $\Sigma$, with $\pi\circ\varphi'_t$ the required isotopy from $\F'$.

We construct $\psi_s$ by the Moser method, which also underlies the Flexibility Theorem and our various local models. Decompose the $\R$-invariant contact form as $\alpha_1=\beta+udz$, where $\beta$ and $u$ are, respectively, a 1-form and function on $\Sigma$, and $u^{-1}(0)=\Gamma$. Let $\alpha_s=(\id_\Sigma\times\eta_s)^*\alpha_1=\beta+sudz$. We would like to construct $\psi_s$ on $\Sigma$ so that $(\psi_s\times\id_\R)^*\alpha_s$ is $\alpha_1$ up to scale, for then $\psi_s\times\eta_s$ preserves $\alpha_1$ up to scale so is a contactomorphism. By direct computation, $\alpha_s\wedge d\alpha_s=s\alpha_1\wedge d\alpha_1$ (since $\beta\wedge d\beta\in \Omega^3(\Sigma)$ and $dz\wedge dz$ both vanish). Thus, $\xi_s=\ker\alpha_s$ is a contact structure whenever $s>0$. Then $d\alpha_s|\xi_s$ is nondegenerate, so there is a unique, $\R$-invariant, vector field $w_s$ in $\xi_s$ (equivalently, with $\iota_{w_s}\alpha_s=0$) such that $\iota_{w_s} d\alpha_s=-\frac{d\alpha_s}{ds}$ on $\xi_s$, and a unique function $g_s$ on $\Sigma$ so that on all of $T(\Sigma\times\R)$ we have
$$\iota_{w_s} d\alpha_s+\frac{d\alpha_s}{ds}=g_s\alpha_s.$$
Since $\frac{d\alpha_s}{ds}=udz$, evaluating both sides of the above equation on $w_s$ shows that $w_s$ has vanishing $z$-component so must lie in $\Sigma$ along its characteristic line field $\ker\beta$. In particular, it is boundary-parallel. (At singular points where $\beta=0$, $\frac{d\alpha_s}{ds}|\xi_s$ vanishes so $w_s=0$.) Thus, $w_s$ integrates to a global isotopy of the form $\Psi_s=\psi_s\times\id_\R$ with $\psi_1=\id_\Sigma$. Since $\iota_{w_s}\alpha_s=0$,
$$\frac{d}{ds}\Psi^*_s\alpha_s=\Psi_s^*(d\iota_{w_s}\alpha_s+\iota_{w_s} d\alpha_s+\frac{d\alpha_s}{ds})=g_s\Psi_s^*\alpha_s.$$

\noindent Thus, in each cotangent space $T^*_p(\Sigma\times\R)$, $\Psi_s^*\alpha_s$ is a curve whose projection to the unit sphere is constant, so $\Psi_s^*\alpha_s$ is $\alpha_1$ up to scale as required. Along $\Gamma$, $u=0$ and any vector $v_0$ parallel to $\R$ lies in $\xi_s=\ker\beta$, so the first displayed formula shows $d\alpha_s(w_s,v_0)=0$. Since $d\alpha_s|\xi_s$ is nondegenerate and $w_s$ has no $z$-component, $w_s$ vanishes on $\Gamma$. Thus, $\psi_s$ fixes $\Gamma$ as required.
\end{proof}

 As we have seen, the dividing set of a convex torus with a tight neighborhood consists of a nonzero, even number of parallel essential curves. The Flexibility Theorem allows us to arrange a nonsingular foliation with one closed leaf in each component of $R_\pm$. These will be nondegenerate, attracting in $R_-$ and repelling in $R_+$, and their orientations can be chosen arbitrarily, allowing a variable number of Reeb components. Nondegenerate means that the derivative of the return map near each closed leaf is not 1, so small perturbations of the line field do not change the leaf structure of the foliation. We especially need the simplest case:

\begin{de}\label{stdForm} We will call a foliation on a torus $\Sigma$ {\em simple} if it has exactly two closed leaves, both nondegenerate and horizontal in some parametrization $\Sigma\approx \R^2/\Z^2$ for which the oriented line field has positive $x$-component everywhere. Its {\em direction} is the (primitive) class of a closed leaf in $H_1(\Sigma)$.
\end{de}

\noindent  Simplicity in a fixed direction is stable under small perturbations of the (nonsingular) line field. In a contact 3-manifold, a torus with a simple characteristic foliation is convex with the above form with two dividing curves, both horizontal, and no Reeb components.

\subsubsection{Legendrian realization}\label{Lerp} Another important application of the Flexibility Theorem is the {\em Legendrian Realization Principle}. Given a convex surface $\Sigma$ with $\partial\Sigma$ empty or Legendrian, suppose $C$ is a compact 1-manifold (not necessarily connected) in $\inter\Sigma$, intersecting $\Gamma$ transversely, and with $\partial C\subset\Gamma$. Then after an isotopy as in the theorem, we can assume $C$ is Legendrian with a standard local model in $\Sigma$, provided that it is {\em nonisolating}. (This means that every component of $\Sigma-(\Gamma\cup C)$ has closure intersecting $\Gamma$.) In our applications, $C$ will be an arc or a nonseparating circle disjoint from $\Gamma$; these are obviously nonisolating. The proof is to construct a suitable $\F$ and apply the Flexibility Theorem, so the addendum automatically applies. This principle gives us a way to add components to the dividing set of $\Sigma$ by {\em folding} (e.g.~\cite[Section~5.3.1]{H}). If $C$ is a nonisolating circle disjoint from $\Gamma$, we can assume it is Legendrian. Since it is disjoint from $\Gamma$, $\xi|C$ projects to $T\Sigma|C$ with the same sign everywhere. Consider an isotopy of $\Sigma$ supported near $C$ and fixing it, but rotating its normal vectors to reverse that sign everywhere on $C$. The resulting surface is still convex (for a new contact vector field), but has two more components of its dividing set since a tubular neighborhood of $C$ has switched between $R_\pm$. Clearly, there must be a nonconvex intermediate stage.

\subsubsection{Bypasses in contact 3-manifolds}\label{BypassCont} One of the main tools of Honda's approach toward classifying contact 3-manifolds is the {\em bypass} operation~\cite{H}. Let $C$ be an embedded arc in a convex surface $\Sigma$ as above, with $C$ intersecting $\Gamma$ transversely in exactly three points including $\partial C$. By the Legendrian Realization Principle (Section~\ref{Lerp}), we may assume $C$ is Legendrian without disturbing $\Gamma$ (after an isotopy of $C$ in $\Sigma$ and a $C^0$-small isotopy of $\Sigma$, by Addendum~\ref{C0}). Suppose we can find a half-disk $D$ transverse to $\Sigma$ with $D\cap \Sigma=C \subset \partial D$, such that the rest of $\partial D$ is a Legendrian curve transverse to $\Sigma$. Also suppose that $tb(\partial D)=-1$ and that $D$ has a tight neighborhood. Then we can make $D$ convex, and its dividing set will be a single arc with boundary on $\inter C$. This {\em bypass} $D$ has a neighborhood that can be given a standard model via the Flexibility Theorem, again by a $C^0$-small isotopy. (The isotopy of the addendum fixes the two corners of $D$ since the characteristic foliation is singular there, with $\beta$ and $w_s$ vanishing.) In the model, we can push $\Sigma$ across $D$. The resulting surface is again convex, but its dividing set has changed as in Figure~\ref{bypassFig} (where the bypass $D$ extends out of the page in the left diagram). This allowed Honda to simplify convex tori. For example, if the three horizontal curves in the figure belong to three distinct circles of $\Gamma$, the bypass move merges them into a single circle. For a torus whose dividing set consists of only two essential circles, the move performs a Dehn twist on $\Gamma$ (Figure~\ref{DehnBypass}). The main difficulty in applying such moves is finding the required bypass half-disks. Even when these half-disks do exist, they cannot be made arbitrarily small. Perhaps surprisingly, the situation is much different in Engel manifolds (Section~\ref{BypassSub}).

\begin{figure}
\labellist
\small\hair 2pt
\pinlabel {$C$} at 52 23
\endlabellist
\centering
\includegraphics{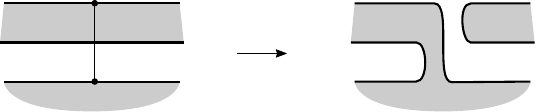}
\caption{A bypass, with shading denoting $R_+$ or $R_-$.}
\label{bypassFig}
\end{figure}

\begin{figure}
\labellist
\small\hair 2pt
\pinlabel {$C$} at 42 21
\endlabellist
\centering
\includegraphics{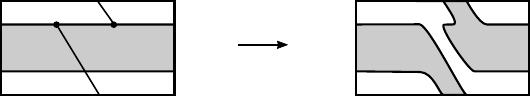}
\caption{A bypass on a torus changing $\Gamma$ by a Dehn twist.}
\label{DehnBypass}
\end{figure}

\subsection{Engel topology}\label{Engel}

An {\em Engel manifold} is a 4-manifold $M$ with a maximally nonintegrable 2-plane field $\D$. That is, $\E=[\D,\D]$ is 3-dimensional everywhere and is itself an even-contact structure, $[\E,\E]=TM$. By the discussion at the beginning of Section~\ref{Background}, $\E$ is the kernel of a 1-form $\alpha$, and $d\alpha(u,v)=\alpha[v,u]$ for vector fields in $\E$. Then $d\alpha|\D=0$ since $[\D,\D]$ vanishes mod $\E$. Since $d\alpha$ is well-defined and nowhere zero on $\E/\W$, the canonical line field $\W$ must then lie in $\D$. Thus, we obtain a canonical flag $\W\subset\D\subset\E\subset TM$. We assume these distributions are compatibly oriented by conventions discussed in Section~\ref{Orientation}.

\subsubsection{Prolongations}\label{Prolong}
The classical example of an Engel manifold is the Cartan {\em prolongation} $\pr (N,\xi)=\pr N$ of a contact 3-manifold $(N,\xi)$. Here $\p\co \pr N\to N$ is the unit circle bundle in the plane bundle $\xi\to N$, so a point $p\in \pr N$ is a point $\p(p)\in N$ together with an oriented line $\L_{\p(p)}$ in $\xi_{\p(p)}$. We take $\D_p$ to be the unique oriented 2-plane in $T_p\pr N$ projecting onto  $\L_{\p(p)}$. Then $\W$ is tangent to the fibers of $\pr N$, and $\E$ projects onto $\xi$, exhibiting the canonical contact projection of the even-contact structure $\E$ as described at the beginning of Section~\ref{Background}. For each $p\in \pr N$, we can identify a neighborhood of $\p(p)$ in $N$ with the standard contact manifold $(\R^3,dz+xdy)$. We then have $p\in\pr\R^3\subset\pr N$. We identify $\pr\R^3$ with $S^1\times\R^3$ by taking the $S^1$-coordinate $w$ of each $q\in\pr\R^3$ to be {\em minus} the polar coordinate determining the line $\L_{\p(q)}$ in $\xi_{\p(q)}$,  measured by projecting $\xi_{\p(q)}$ to the $xy$-plane. (We order the coordinates $(w,x,y,z)$, as justified in Section~\ref{Orientation}. Then the sign of $w$ is analogous to that of $-x=\frac{dz}{dy}$ for front projections in $\R^3$.) Now $\E$ is the kernel of $\alpha=dz+xdy$ and $\D=\ker\alpha\cap\ker\beta$, where $\beta=\sin(w)dx+\cos(w)dy$. Note that $\beta$ is a positive contact form on each slice with constant $z$ (although it is not canonical in an Engel manifold the way $\alpha$ is). The projection $\p$ simply forgets $w$, so $\W$ is tangent to the circles with $(x,y,z)$ constant.

\subsubsection{Local models of $\W$-transverse 3-manifolds}\label{LocalEngel}
We saw at the beginning of Section~\ref{Background} that any 3-manifold $N$ embedded $\W$-transversely in an even-contact 4-manifold $M$ inherits a contact structure $\xi=\E\cap TN$ that is preserved by flows along $\W$. If $M$ is itself a prolongation, its projection restricts to a local contactomorphism on any such $N$. For $M$ an arbitrary Engel manifold, $N$ also inherits an oriented line field $\L=\D\cap TN\subset\xi$ that rotates under such flows. This line field determines a canonical section $N\to\pr(N,\xi)$ that extends, by a flow along $\W$, to an embedding of a neighborhood of $N\subset M$ into $\pr(N,\xi)$ preserving the Engel structure. Thus, we can model any $\W$-transverse 3-manifold in an Engel manifold as a contact manifold with fourth coordinate represented by a variable line field $\L\subset \xi$, analogously to a front projection as in Section~\ref{Proj}. As an important special case, for any $p\in M$ we can choose $(N,p)\approx(\R^3,0)$ so that $\L_p$ corresponds to the positively oriented $x$-axis. The coordinates near 0 defined above on $\pr\R^3$ can now be used near $p$ in $M$. This exhibits the local equivalence of all Engel manifolds, as required by topological stability.

\begin{ob}\label{tight}
Let $M_\infty\approx\R^4$ be the universal cover of $\pr(\R^3,dz+xdy)$, and for $r\in\R^+$, let $M_r\subset M_\infty$ be the open subset defined by requiring $|w|<r$. Any bounded region of $M_r$ can be squeezed into a region over a preassigned neighborhood of $0\in\R^3$ by an automorphism of the form $(w,x,y,z)\mapsto(w,\epsilon x, \epsilon y,\epsilon^2 z)$. Then the above local coordinates show that every point in an Engel manifold has a neighborhood containing all bounded regions of $M_r$ for sufficiently small $r$. Thus, if every $M_r$ is overtwisted then so is every Engel manifold, so the manifolds $M_r$ with $r$ small are the most natural candidates for tight Engel manifolds and natural venues for transverse-torus 2-knot theory as in Questions~\ref{ques}.
\end{ob}

\subsubsection{Characteristic foliations of surfaces in even-contact 4-manifolds}\label{EngelChar}
 Every surface $\Sigma$ embedded (or immersed) in an even-contact 4-manifold $M$ inherits a singular {\em characteristic line field} $\E\cap T\Sigma$, whose singularities occur at tangencies of $\Sigma$ with $\E$. This can be singular everywhere. (Given a Legendrian knot $K$ in a contact manifold $N$, consider the torus in $\pr N$ made by restricting the circle bundle to $K$.) However, generically there are only finitely many such tangencies. Away from these, the resulting {\em characteristic foliation} (or {\em $\E$-foliation}) is oriented by an orientation of $\Sigma$ (via the orientations on $\E$ and $M$; see Section~\ref{OrientL}). If $\Sigma$ is $\E$-transverse, it must be a torus since the foliation trivializes its tangent bundle. Any surface $\Sigma\subset M$ without $\W$-tangencies can be extended to a 3-manifold $N\approx\Sigma\times \R$ transverse to $\W$, which canonically inherits a contact structure $\xi=\E\cap TN$. Then the $\E$-foliation on $\Sigma$ agrees with the {\em $\xi$-foliation}, the characteristic foliation in the sense of contact topology (Section~\ref{SurfCont}). Since the germ of $(N,\xi)$ is preserved by flows along $\W$, it is uniquely determined (and also obtained by projecting out the flow). Thus, concepts such as convexity and dividing sets canonically extend to surfaces without $\W$-tangencies in even-contact 4-manifolds.

\subsection{Orientation conventions}\label{Orientation}
For most of our discussion, it suffices to know that orientation conventions exist, but we now describe our conventions carefully for completeness, and to determine some signs of invariants in Sections~\ref{Dnu} and \ref{DT}. These conventions are chosen for optimal compatibility with the standard orientations of smooth and contact topology. We first orient the flag of $\pr\R^3$ for the standard contact structure on $\R^3$, then pass to the general case via the local coordinates from Section~\ref{LocalEngel} (which depend on two binary choices if the orientations are not prespecified).

\subsubsection{The local model} As in Section~\ref{Prolong}, we identify $\pr\R^3$ with $S^1\times\R^3=\R^4/2\pi\Z$. We write the $S^1$-factor first so that when a $\W$-transverse 3-manifold $N$ flows by a vector field in the positive $\W$-direction, the leading boundary of the resulting embedded $I\times N$ has the same orientation as $N$, by the standard ``outward normal first'' convention. As in Section~\ref{Prolong}, we describe the flag using the forms $\alpha=dz+xdy$ and $\beta=\sin(w)dx+\cos(w)dy$. (The other common convention $\alpha'=dz-ydx$ for the standard contact $\R^3$ merely differs from this by a $\pi/2$-rotation of the $xy$-plane. Our convention has the advantage that front projection is the obvious projection into the page if we draw the axes as usual with the $x$-axis pointing out of the page. Common conventions for the standard Engel chart are obtained from ours by linearizing at $w=n\pi/2$, interpreting $w$ as a slope instead of an angle.) To orient the Engel flag, let $W$ and $Z$ be the unit vector fields parallel to the positively oriented $w$- and $z$-axes, and let $X$ and $Y$ be the lifts to $\xi$ in each $\{w\}\times\R^3$ of the vector fields $(\cos(w),-\sin(w))$ and $(\sin(w),\cos(w))$, respectively, in the $xy$-plane. Then $(W,X,Y,Z)$ is the standard ordered basis when $w=x=0$. Everywhere, we have $\alpha(Z)=\beta(Y)=1$ and otherwise these vector fields lie in the kernels of these forms. Thus, the first $k$ of these vector fields span the $k$-dimensional distribution in the flag $\W\subset\D\subset\E\subset TM$. We orient each distribution, and hence, each of the ten  quotients of pairs, using the corresponding subset of vector fields in the given order. (This orients $\W$-transverse 3-manifolds, $\D$-transverse surfaces and $\E$-transverse curves by putting their normal orientations first to compare with the ambient orientation. In general, we orient quotients by writing the kernel first. We write transverse intersections so that the normal orientations of the intersection within each of the two factors, in the given order, followed by the  orientation of the intersection, give the ambient orientation. For example, $\xi=\E\cap T\R^3=-T\R^3\cap\E$ exhibits $TM$ as ${\rm span}(W,Z,X,Y)$. The order matters when both factors have odd codimension.) It seems reasonable to use alphabetical order for the span of any subset of these vector fields, e.g., $\ker\beta={\rm span}(W,X,Z)$, although we do not need these additional cases. The rest of this section transfers the above choice of ten orientations to arbitrary Engel manifolds, and examines the relevant consequences of this choice and why it is especially natural.

\subsubsection{The general case} Two of our ten quotients are canonically oriented in all Engel manifolds (even without our standing orientability condition) by a universal choice. Recall from the beginning of Section~\ref{Background} that for vector fields $u,v$ in a hyperplane field $\ker\alpha$, $d\alpha(u,v)=\alpha[v,u]$. Contact 3-manifolds are then canonically oriented by the ordered local bases
$$(u,v,[v,u])$$
for any local bases $(u,v)$ of $\xi$ (whose choice does not affect the result). This is the standard convention for contact 3-manifolds $N$, since choosing the sign of $\alpha$ orients $\xi$ (through $d\alpha$) and $TN/\xi$ compatibly, so that their oriented sum is $TN$. In Engel manifolds $M$, we use the same convention to canonically orient $TM/\W$, where $(u,v)$ is any basis for $\E/\W$. This implies $\W$-transverse 3-manifolds $N$ inherit their usual contact orientation by the identification $TN\cong (TM/\W)|N$, and $\alpha$ orients $\E/\W$ and $TM/\E$ compatibly. We canonically orient $\E=[\D,\D]$ by the same convention with $u$, $v$ in $\D$, and $\alpha$ replaced by $\beta$ given by any local coordinates as above. (The opposite convention can be found in the literature, but ours has the advantages that it orients $\E$ so that any local hypersurface tangent to $\E$ at a point inherits its $\beta$-contact orientation from $\E$, and choosing the sign of $\beta$ orients $\D$ and $\E/\D$ compatibly with this.) Note that our orientation convention for $\pr\R^3$ in the previous paragraph agrees with these canonical orientations on $\E$ and $TM/\W$ (since $d\beta(W,X)=1=\beta(Y)$).

Now we can analyze the remaining orientations in a general Engel manifold. Since $TM/\W$ and $\E$ are now canonically oriented in all Engel manifolds, orienting $TM$ is equivalent to orienting $\W$ or $\E/\W$, the latter corresponding to $\xi$ in any $\W$-transverse 3-manifold. This choice and the orientation of $\D$ are arbitrary: There is a pair of involutions on $\pr\R^3$ defined by reversing the signs of coordinates $w,x,z$ or $x,y$, respectively. These preserve the Engel structure but independently reverse the two orientations in question, as well as the signs of the forms $\alpha$ and $\beta$. Independently choosing the above pair of orientations allows four possible ways to orient the flag, orienting all ten quotients. We always assume these choices have been made, with $\alpha$ chosen to evaluate positively on $TM/\E$. We then choose our local coordinates so that the orientations are described by our convention for $\pr\R^3$, and $\beta$ evaluates positively on $\E/\D$.

\begin{Remarks}\label{Klein}
(a) If we allow nonorientable Engel flags, we can construct transverse Klein bottles as quotients of the $yz$-plane in $\pr\R^3$, and arrange either $TM$ or the normal bundle (hence $\D$) to be orientable (but not both). For the former, mod out unit $y$-translation, and then unit $z$-translation composed with the involution reversing the signs of $x,y$. For the latter, first mod out $z$-translation, then $y$-translation composed with reversing $w,x,z$.

\item[b)] More generally, the above flexibility of orientations allows Engel structures with various forms of nonorientability. Up to homotopy, every formal Engel structure (suitably defined in the presence of nonorientability) can be made into a loose \cite{CPP} or overtwisted \cite{PV} Engel structure. Prolongations supply concrete examples:

\item[i)] Prolongations of nonorientable contact plane fields (defined as in Section~\ref{Prolong} but with  $\alpha$ globally defined only with values in a twisted line bundle) have $\W$ nonorientable but $\D/\W$ canonically oriented, so $\D$ and $TM$ are also nonorientable with $$w_1(\D)=w_1(\W)=\p^*w_1(\xi)=w_1(TM)\ne0.$$
\item[ii)] Every prolongation has a free involution obtained by reversing the orientation of the lines $\L$, i.e., the fiberwise antipodal map. The quotient has nonorientable $\D$ with $w_1(\D)$ having nonzero value on the fibers, but if $\xi$ is orientable then so is $TM$, and $$w_1(\D)\ne0, \ \ w_1(TM)=0.$$
\item[iii)] If, instead, $\xi$ and hence $TM$ are nonorientable in (ii), the previous computation still applies near a fiber, so $$0\ne w_1(\D)\ne w_1(TM)\ne0.$$ Thus, we have exhibited three different types of behavior of these classes. The example in (a) with $\D$ orientable but $TM$ nonorientable exhibits the remaining case $$w_1(\D)=0, \ \ w_1(TM)\ne0$$ for Engel structures exhibiting nonorientability.
\end{Remarks}

\subsubsection{Surfaces in contact 3-manifolds}\label{OrContSurf} The characteristic foliation $\F$ on a surface $\Sigma$ in a contact 3-manifold canonically inherits an orientation from orientations of $\Sigma$ and $\xi$. The standard convention is that a vector positively tangent to $\F$ (so in $\ker\alpha$) followed by a vector on which $\alpha$ is positive should be a positive tangent basis for $\Sigma$. This convention is chosen so that positive singularities (at which $\xi_p=T_p\Sigma$ as oriented planes) have positive divergence. (The divergence is defined using any vector field on $\Sigma$ that when contracted with some positive area form for $\Sigma$ yields $\alpha|\Sigma$. Such a vector field is positively tangent to the nonsingular part of $\F$.) It ultimately follows that in the convex case, $R_+$ (bounded by the dividing set with $\F$ directed transversely outward) contains all of the positive singularities and none of the negative ones. As an example of orienting a characteristic foliation, suppose $\Sigma$ bounds a tubular neighborhood of a transverse knot $K$ as in the proof of Proposition~\ref{helices}. Then $K$ is canonically oriented (so that $\alpha|K$ is positive), and $\Sigma$ is foliated by helices that are left-handed relative to a given longitude $\lambda\subset\Sigma$ when the tubular neighborhood is sufficiently narrow. It is natural to orient $\lambda$ to be parallel to $K$, and a meridian $\mu\subset\Sigma$ to be right-handed (linking $K$ positively). Then $\mu\cdot\lambda=1=-\lambda\cdot\mu$ in $\Sigma$, where $\Sigma$ is oriented as the boundary of the tubular neighborhood with its contact orientation. The standard orientation convention now orients the foliation toward the quadrant determined by $\mu$ and $-\lambda$.

\subsubsection{Surfaces in Engel manifolds}\label{OrientL} Now suppose $\Sigma$ is a surface in an even-contact 4-manifold $M$. If $\Sigma$ has no $\W$-tangencies, it canonically lies in a contact 3-manifold $N$, and its $\E$-foliation corresponds to the $\xi$-foliation in $N$ (Section~\ref{EngelChar}). We would like these two characterizations of the foliation to determine the same orientation. This is achieved by modding out $\W$ and applying the previous paragraph. Thus, a positive tangent vector to the $\E$-foliation, followed by a tangent vector to $\Sigma$ on which $\alpha$ is positive should again give the preassigned orientation on $\Sigma$. If $\Sigma$ is transverse to an Engel plane field $\D$, then its orientation is inherited by identifying $\D|\Sigma$ as the normal bundle and $(TM/\D)|\Sigma$ with $T\Sigma$. Modding out $\W$ identifies this orientation with the one on $\Sigma$ in $N$ for which the oriented line field $\L=\D/\W$ is positively transverse. For example, if $\Sigma$ projects to a boundary in $N$, preserving orientation, then $\L$ must be outward transverse. Alternatively, in our local coordinates, $\beta\wedge\alpha$ is a positive area form on $TM/\D$ and hence on $\Sigma$ in $M$, so $\beta$ is positive on the characteristic line field of a transverse surface. This is also consistent with identifying the characteristic line field as $\E/\D\subset TM/\D$ and locally applying our convention orienting these by $Y$ and $(Y,Z)$.


\section{Making surfaces transverse}\label{Main}

To prove Theorem~\ref{main0}, we need to isotope a given torus $\Sigma$ in an Engel manifold to be transverse. Since $\Sigma$ typically has regions that are already transverse but with incompatible orientations, we do not attempt to achieve transversality directly, but instead adapt the notion of transverse pushoffs of Legendrian knots  from Section~\ref{KnotOps}. The corresponding notion of Legendrian tori and their transverse pushoffs is introduced in Section~\ref{Surface}, along with some useful examples. To make $\Sigma$ Legendrian, we first need to control its characteristic foliation, which we do in the setting of even-contact 4-manifolds (Section~\ref{Bypass}). Our main tool for this is the bypass operation from Section~\ref{BypassCont}, which can be done more easily in this 4-dimensional setting (Section~\ref{BypassSub}). The proof of Theorem~\ref{main0} is then completed in Section~\ref{MainProof}.

\subsection{Surfaces in Engel manifolds}\label{Surface} 
Recall from Section~\ref{Knots} that in a contact 3-manifold $(M,\xi)$, a knot $K$ is Legendrian (resp.\ transverse) if $\dim(\xi\cap TK)$ is 1 (resp.\ 0) everywhere on $K$. We would like to similarly understand embedded surfaces suitably compatible with an Engel structure. In contrast with Legendrian knots, a surface $\Sigma$ cannot be tangent to $\D$ on an open subset of $\Sigma$, since $\D$ would be closed under Lie bracket there. Thus, the two corresponding extremal cases are given by the following definitions, which we discuss in the two subsequent sections.

\begin{de}\label{Leg}
A surface $\Sigma$ in an Engel manifold $M$ is {\em Legendrian} if $\dim(\D\cap T\Sigma)\ge 1$ everywhere and {\em transverse} (or {\em $\D$-transverse}) if the dimension is zero everywhere.
\end{de}

Kegel \cite{K} uses the term ``Legendrian'' for surfaces that are tangent to an even-contact structure. These could be called {\em $\E$-Legendrian} when necessary to avoid confusion with the {\em $\D$-Legendrian} surfaces satisfying our definition above. We will have no use of the former notion. In fact, we need our Legendrian surfaces to be transverse to $\E$.

\subsubsection{Legendrian surfaces} Our present interest in Legendrian surfaces is as a tool for constructing transverse surfaces. Thus, we make no attempt at a general study of the former, but merely note that Legendrian surfaces can have complicated interactions with the Engel flag. As a nongeneric example, over a Legendrian knot in a contact manifold $N$, the restriction of the circle bundle comprising the prolongation $\pr N$ (Section~\ref{Prolong}) is a Legendrian torus that satisfies $\W|\Sigma\subset T\Sigma\subset\E|\Sigma$ and is tangent to $\D$ along a pair of sections. An $\E$-transverse but still nongeneric example is given by the circle bundle over a transverse knot in a prolongation (Example~\ref{vert}). For an $\E$-transverse Legendrian surface, $\E\cap T\Sigma=\D\cap T\Sigma$ is 1-dimensional, so the characteristic line field is nonsingular. Then the tangent and normal bundles are both trivialized by $\D$, so such a surface must be a torus with trivial normal bundle. Generically, we expect such surfaces to be tangent to $\W$ along 1-manifolds (unlike the previous example) since $\W$ is given by a section of the unit circle bundle of $\D$.

\begin{example}\label{liftReeb}
In prolongations, there is a useful method for generating $\E$-transverse Legendrian tori. Let $\Sigma$ be any torus immersed in a contact 3-manifold $N$ with nonsingular characteristic foliation. (For example, any embedded torus with a tight neighborhood can be isotoped to such a form by the Flexibility Theorem~\ref{flex}.) Any such $\Sigma$ has a unique $\E$-transverse  {\em Legendrian lift} to an immersed torus in the prolongation of $N$: The given orientations on $\Sigma$ and $\xi$ orient the characteristic line field $\L$ in $\xi$, which we then interpret as a section of the unit circle bundle $\pr N$ of $\xi$ over $\Sigma$, cf.\ Section~\ref{LocalEngel}. (Reversing the orientation of $\Sigma$, and hence $\L$, changes the lift by an isotopy: a $\pi$-rotation around the circle fibers of $\pr N$.) The Legendrian lift is embedded whenever $\Sigma$ is. (Generically, embedding only fails at finitely many points where $\xi$ is tangent to a double curve of $\Sigma$.) Beware that small changes in $\Sigma$ can change the homotopy class of the lift. To understand this, first notice that $\xi$-transversality of $\Sigma$ is essential. At a generic singularity, the characteristic line field has degree $\pm 1$, resulting in a puncture where the lift is bounded by a fiber of the prolongation. As an example of a ``small'' change of $\Sigma$, suppose $\Sigma$ is convex in $N$, and its characteristic foliation has a repelling closed leaf $L$. By the Flexibility Theorem~\ref{flex}, we can change the foliation near $L$, after which $L$ is still repelling but is oppositely oriented. We can do this through a 1-parameter family of convex surfaces, although singularities in the foliation will transiently appear. The resulting Legendrian lifts are not homotopic. To see this, consider a local model intermediate state: the $xy$-plane in $(\R^3,dz+xdy)$ mod unit $y$-translation, with $L$ given by the $y$-axis. Then $L$ consists entirely of singular points, with the rest of the foliation parallel to the $x$-axis and oriented outward from $L$. Perturbing this plane near the $y$-axis by rigidly rotating a small neighborhood slightly about the $y$-axis changes the latter to a repelling leaf of a nonsingular foliation. However, the direction of the leaf depends on the direction of the small rotation. When we lift the unperturbed plane to $\pr N$, the plane rips along the $y$-axis, with the two edges lifting antipodally (corresponding to the outward orientations of the foliation on the two half-planes). Each of the two perturbations fills in the Legendrian lift with an annulus spanning the gap. But the two use annuli on opposite halves of the fibers, as seen by examining the directions of the characteristic lines in the two cases. Thus, the homotopy classes of the two Legendrian lifts differ by $\PD[L]$ in $H^1(\Sigma)\cong[\Sigma,S^1]$. In contrast, every isotopy through embedded tori with nonsingular foliations determines an isotopy of their Legendrian lifts.
\end{example}

\subsubsection{Transverse surfaces}\label{TransSurf} As with $\E$-transverse Legendrian surfaces, a transverse surface $\Sigma$ must be a torus with trivial normal bundle. This is because its normal and tangent bundle $\nu\Sigma\cong\D|\Sigma$ and $T\Sigma\cong(TM/\D)|\Sigma$ are trivialized by the Engel flag. These isomorphisms also canonically orient $\Sigma$ and its normal bundle (analogously to transverse knots in contact 3-manifolds, Section~\ref{Knots}). Unlike Legendrian surfaces, a transverse surface can never be tangent to $\W$ or $\E$. In particular, its characteristic line field $\E\cap T\Sigma$ is nonsingular and canonically oriented (via the canonical orientation of $\Sigma$). Just as transverse knots can be constructed as pushoffs of Legendrian knots (Section~\ref{KnotOps}), we have the following:

\begin{prop}\label{push}
Let $\Sigma$ be an $\E$-transverse Legendrian torus in an Engel manifold $M$. Then there is an embedding $\R\times \Sigma\to M$ such that $\{t\}\times \Sigma$ agrees with $\Sigma$ when $t=0$ and is transverse otherwise, with induced orientation depending on the sign of $t$.
\end{prop}

\noindent In particular, there are transverse tori of both orientations $C^\infty$-close to $\Sigma$. For a preassigned orientation on $\Sigma$, we denote the corresponding {\em transverse pushoff} by $\tau \Sigma$. The characteristic line field on $\{t\}\times \Sigma$ is $C^\infty$-close to that of $\Sigma$, although its orientation depends on the sign of $t$.

\begin{proof}
The nonsingular characteristic line field $\D\cap T\Sigma=\E\cap T\Sigma$ is orientable. Thus, near $\Sigma$ in $M$, there is a basis $(u,v)$ for $\D$ where $v|\Sigma$ lies in $\D\cap T\Sigma$. Since $[\D,\D]=\E$, $[u,v]$ is never 0 in $\E/\D$. (Recall that for $\beta$ as in Section~\ref{Prolong} and $u,v$ in $\D\subset\ker\beta$, $\beta[u,v]=d\beta(v,u)$ is determined pointwise.) In particular, $[u,v]$ never lies in the span of $T\Sigma$ and $u$, whose intersection with $\E$ is $\D$. But $[u,v]$ is the Lie derivative of $v$ by $u$. Thus, as $\Sigma$ flows by $u$, it immediately becomes transverse to $\D$, with orientation depending on the sign of $t$.
\end{proof}

\begin{example}\label{transLift}
Given a torus $\Sigma$ embedded with nonsingular characteristic foliation in a contact 3-manifold $(N,\xi)$, we can define its {\em transverse lift} to $\pr N$ by applying Proposition~\ref{push} to its Legendrian lift from Example~\ref{liftReeb}. The vector field $u$ can be chosen to point positively along $\W$, so that the push is a small fiberwise rotation in $\pr N$ (clockwise, by the convention of Section~\ref{Prolong}). This perturbs $\L$ slightly around $\xi$ in the positive transverse direction to $\Sigma$ (Section~\ref{OrientL}). By a further push along the fibers, we can reach any $\L$ in $\xi$ positively transverse to $\Sigma$, so the lift is characterized up to vertical transverse isotopy by requiring $\L$ to be positively transverse to $\Sigma$. (In particular, this defines the transverse lift without using the proposition.) The same construction still gives a transverse embedding if $\Sigma$ is generically immersed  with all double curves transverse to $\xi$, except that the characterization requires $\L$ to be sufficiently close to the tangent planes of $\Sigma$. (See the next remark.) Since a transverse lift is vertically isotopic to the corresponding Legendrian lift, its homotopy class can change as in Example~\ref{liftReeb} when $\Sigma$ is isotoped through convex surfaces with transient singularities. However, isotopies preserving nonsingularity of the foliation can be lifted.
\end{example}

\begin{Remark}
Embedded transverse tori in $\pr N$ can be characterized in $N$: Such a torus $\hat\Sigma$ has no $\W$-tangencies, so it projects to an immersed torus $\Sigma$ in $N$. Then $\hat\Sigma$ can be recovered from the induced line field $\L$ in $\xi|\Sigma$, which is transverse to $\Sigma$. (This transversality implies the characteristic foliation is nonsingular.) At each double point of $\Sigma$, the two (oriented) lines of $\L$ are different since $\hat\Sigma$ is embedded. Conversely, any such pair $(\Sigma,\L)$ determines a transverse torus in $\pr N$. When $\Sigma$ is embedded, its transverse lift is vertically transversely isotopic to the original $\hat\Sigma$, but this need not be true for immersions: If two sheets of $\Sigma$ intersect in a circle $C$, the two line fields $\L$ along $C$, if never antiparallel, will span $\xi|C$. Then perturbing the transverse line fields near $C$ can push them past each other, reversing the orientation on $\xi|C$ that they determine. This has the effect of reversing the order of the two sheets in the $w$-direction (a crossing change in the directions transverse to $C$). This can often change the topological knot type of $\hat\Sigma$, as can be exhibited in Example~\ref{PV}(a) by ranging $K$ over transverse knots with the same front projection.
\end{Remark}

\subsection{Controlling the characteristic foliation}\label{Bypass}

To make a surface transverse, we first need to control its characteristic foliation. We will do this in the setting of even-contact 4-manifolds by adapting the bypass operation from contact 3-manifold topology (Section~\ref{BypassCont}). As a preliminary step, we first arrange the surface to be convex in some $\W$-transverse 3-dimensional submanifold, which is easy if there are no $\W$-tangencies.

\subsubsection{Eliminating $\W$-tangencies}

\begin{lem}\label{noW}
Let $\Sigma$ be a surface in a 4-manifold $M$ with a line field $\W$. Then $\Sigma$ can be isotoped to avoid tangencies with $\W$ if and only if the normal bundle $\nu\Sigma$ is trivial. The isotopy can be assumed $C^0$-small.
\end{lem}

\begin{proof}  Assume $\Sigma$ is generic. Then a nowhere-zero vector field in $\W$ projects to a section of $\nu \Sigma=(TM|\Sigma)/T\Sigma$ whose (isolated) zeroes correspond to the tangencies of $\W$ with $\Sigma$, with signed count given by the Euler number $e(\nu\Sigma)$. If these can be eliminated by an isotopy of $\Sigma$, then $\nu\Sigma$ is trivial. For the converse, we assume $e(\nu\Sigma)=0$ and cancel zeroes in pairs of opposite sign. To do this, we first consider a local model $f(u,v)=(w,x,y,z)$ of the embedding with a $\W$-tangency at 0, where $\W$ is parallel to the $w$-axis. After projecting out $\W$, the resulting map $\R^2\to\R^3$ must be an immersion except at 0, where we have a generic nonimmersed point, a Whitney umbrella. Two of these are shown in Figure~\ref{Whitney}. A local model is given by $f(u,v)=(v,u,uv,v^2)$, which is an embedding in $\R^4$ whose quotient in $\R^3$ (by dropping the $w$-coordinate) is the model Whitney umbrella. The latter can be visualized as the set of horizontal (constant $z$) lines cutting through the $z$-axis and the parabola $(1,v,v^2)$, so there is a double curve ($u=0$) along the positive $z$-axis. If the lift to $\R^4$ is sliced along a hyperplane $z=c^2>0$, we see a pair of skew lines $ (w,x,y)=(\epsilon c, u,\epsilon cu)$ with $\epsilon=\pm1$; these merge together as $c\to 0$. Note that the reflection of $\R^4$ reversing the sign of $w$ interchanges the levels of the two skew lines (which looks like a crossing change when we also project out $w$). This exhibits a mirror image pair of surfaces projecting to the same Whitney umbrella, corresponding to the two signs of zeroes of the section of $\nu\Sigma$. Now any two $\W$-tangencies in $\Sigma$ can be connected by a smoothly embedded curve $C$ in $\Sigma$ avoiding any other tangencies. We can choose $C$ to appear in a local model as any radial ray in its domain, such as either half of the $u$-axis. Note that the upward normal of $\Sigma$ in the Whitney umbrella along the positive $u$-axis corresponds to the downward normal along the negative $u$-axis. Thus, we can choose the rays so that a tubular neighborhood of $C\subset \Sigma$ is modeled in $\R^3$ by a segment of the $x$-axis with two concave-up Whitney umbrellas, smoothly joined along a common constant-$x$ parabola (Figure~\ref{Whitney}). Since there are no other nearby $\W$-tangencies, we can pass to $\R^4$ by including the $w$-coordinate in the obvious way. We can now cancel the two $\W$-tangencies if and only if they have opposite sign. To see this, again slice along the hyperplanes $z=c^2$. Now, projecting out the $w$-axis gives a pair of curves with two crossings. If the signs are the same, the curves are linked in $\R^3$ (Figure~\ref{Whitney}(b)). But if the signs are opposite (Figure~\ref{Whitney}(c)), we can eliminate the crossings by a Type-II Reidemeister move for small $|c|$. This eliminates two $\W$-tangencies by a $C^0$-small isotopy suported in the model. (Note that the isotopy is smooth along the $u$-axis since the latter remains fixed and its normal vectors rotate in the $wy$-plane. An alternative description is to let $x$ be time. The two Whitney umbrellas together look like a parabola in $\R^3$ that pivots across the $wz$-plane and then pivots back, creating two opposite $\W$-tangencies. Our isotopy simply tapers this to the identity for small $|v|$.)
\end{proof}

\begin{figure}
\labellist
\small\hair 2pt
\pinlabel {a)} at 73 163
\pinlabel {b)} at 58 35
\pinlabel {c)} at 223 35
\pinlabel {$x$} at 95 72
\pinlabel {$y$} at 184 62
\pinlabel {$z$} at 193 173
\pinlabel {$z$} at 193 173
\pinlabel {$z=c^2$} at 315 128
\pinlabel {$\Sigma$} at 218 69
\endlabellist
\centering
\includegraphics{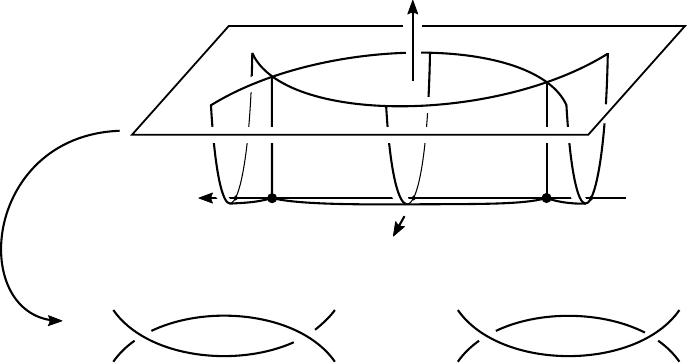}
\caption{(a) A pair of $\W$-tangencies projected into $xyz$-space and (b,c) possible lifts to $wxy$-space of a constant $z$ slice. The pair cancels if and only if the tangencies are oppositely oriented as in (c).}
\label{Whitney}
\end{figure}

As noted in Section~\ref{EngelChar}, if the above $M$ is even-contact, then once $\Sigma$ has no $\W$-tangencies it is canonically embedded in a germ of a contact 3-manifold $N\subset M$. A further $C^\infty$-small perturbation in $N$ then makes $\Sigma$ convex (which is well-defined in $M$). Thus:

\begin{cor}\label{convex}
Let $\Sigma\subset M$ be a surface with trivial normal bundle in an even-contact 4-manifold. Then there is a $C^0$-small isotopy making $\Sigma$ convex in $M$ (i.e.\ convex in a $\W$-transverse 3-manifold $N\subset M$). \qed
\end{cor}

\subsubsection{Bypasses in even-contact 4-manifolds}\label{BypassSub} While the bypass operation is a powerful tool in contact topology (Section~\ref{BypassCont}), it can be difficult to find the required half-disks, and the resulting isotopies cannot be made $C^0$-small. In contrast, we now show that suitably small half-disks can always be found in even-contact 4-manifolds. We then easily do our required simplification of characteristic foliations, realizing any preassigned simple foliation (Corollary~\ref{standardize}).

\begin{lem}\label{bypass}
Let $\Sigma$ be a convex surface with dividing set $\Gamma$ in an even-contact 4-manifold $M$. Let $C$ be an arc embedded in $\Sigma$ that transversely intersects $\Gamma$ at its endpoints $\partial C$ and exactly one interior point. Then after an isotopy of $C$ in $\Sigma$ rel $\Gamma$, there is a $C^0$-small isotopy of $\Sigma$ in $M$,  after which $\Sigma$ is still convex, but with $\Gamma$ changed by a bypass on $C$ (Figure~\ref{bypassFig}, or if preferred, its mirror image).
\end{lem}

\begin{proof}
By hypothesis, $\Sigma$ is convex in a $\W$-transverse $N$. By the Legendrian Realization Principle using Addendum~\ref{C0}, which isotopes $C$ in $\Sigma$ and $\Sigma$ in $N$ as allowed by the lemma (Section~\ref{Lerp}), we can make $C$ Legendrian with a neighborhood in $N$ given by the following standard model: In $\R^3$ with the tight contact structure $\xi$ given by $\cos(x) dy-\sin(x)dz$, the $xy$-plane is convex (relative to the unit vector field parallel to the $z$-axis) with dividing set given by the lines $x=n\pi$ for $n\in\Z$. Identify $C$ with the Legendrian arc $[-\pi,\pi]$ on the $x$-axis and $\Sigma$ with the $xy$-plane in a neighborhood of that arc. (To reverse orientation on $\Sigma$, flip the signs of $y$ and $z$.) In the front projection to the $yz$-plane (which is analogous to that of $dz+xdy$ but with more rotation), consider the heart-shaped curve in Figure~\ref{heart}. This lifts to a smooth Legendrian curve $L$ in $\R^3$ with the same endpoints as $C$ (which projects to the origin). Thus, $C\cup L$ bounds an embedded half-disk $D$, which we can assume lies in our given neighborhood of $C$ in $N$, although its interior intersects $\Sigma$. We can at least arrange $D$ to be transverse to $\Sigma$, with outward normal along $C$ pointing downward. It is routine to check from the front projection that $tb(\partial D)=-1$. (There is one left twist in $\xi$ along $C$, and no contribution from $L$ since the two cusps contribute canceling half-twists to a $\xi$-transverse vector field.) Thus, after a small perturbation, $D$ can be identified with a standard model bypass. This fails to be a bypass for $\Sigma$ in $N$ since $D-C$ intersects $\Sigma$. However, a neighborhood of $C$ in $D$ does intersect $\Sigma$ correctly. By flowing the rest of $D$ a small amount away from $N$ along $\W$ in $M$, we obtain an embedding of $\Sigma\cup_C D$ in $M$ with no $\W$-tangencies on either surface. Thickening this union slightly in a direction transverse to $\W$, we obtain a new $N$ (with a locally contactomorphic projection to the old $N$). In this new $N$, we can perform a bypass move on $D$.

\begin{figure}
\labellist
\small\hair 2pt
\pinlabel {$C$} at 80 69
\pinlabel {$z$} at 81 107
\pinlabel {$y$} at 159 82
\pinlabel {$\Sigma$} at 12 82
\pinlabel {$L$} at 105 25
\endlabellist
\centering
\includegraphics{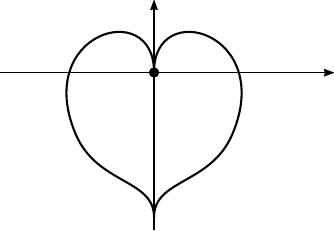}
\caption{A bypass half-disk intersecting $\Sigma$.}
\label{heart}
\end{figure}

To arrange the isotopy to be $C^0$-small, control the $y$-, $z$- and $w$-coordinates by taking the front projection of $D$ and its push along $\W$ sufficiently small. For the $x$-coordinate, which cannot be assumed to be small, note that the tangent angles of the image of $L$ in Figure~\ref{heart} are monotonic, except near the bottom cusp. The reversal of the angle there can be kept small by using a sufficiently narrow diagram. Thus, the $x$-coordinate of $L$ is close to being monotonic, so we can choose our original $D$ to admit a (noncharacteristic) foliation by arcs from $C$ to $L$ whose $x$-coordinates are $C^0$-close to being constant. Our {\em coronary bypass} isotopy can now be chosen $C^0$-small, flowing a neighborhood of $C$ in $\Sigma$ along these arcs.
\end{proof}

\begin{cor}\label{standardize}
Suppose $\Sigma$ is a torus in an even-contact 4-manifold $M$, and that the plane bundles $\nu\Sigma$ and $(\E/\W)|\Sigma$ are trivial. Then after a $C^0$-small isotopy, we may assume $\Sigma$ is convex with a simple foliation (Definition~\ref{stdForm}), preassigned up to isotopy in $\Sigma$ and realizing any preassigned direction in $H_1(\Sigma)$.
\end{cor}

\begin{proof} By Corollary~\ref{convex}, we can assume $\Sigma$ is convex in $M$. Its dividing set $\Gamma$ then splits $\Sigma$ into two regions $R_+$ and $R_-$ (see the beginning of Section~\ref{Convex}), so there must be an even number of homologically essential components of $\Gamma$, all parallel. By folding along a nonisolating circle if necessary, we can arrange $\Gamma$ to have at least two such components. If there are more than two, we can find a bypass arc connecting three of them, and a coronary bypass merges these into a single component. Thus, we can reduce the number of essential components of $\Gamma$ to exactly two. If $R_+$ or $R_-$ has a homologically trivial component with at least two boundary components, we may connect two boundary circles to some other component of $\Gamma$ by an arc; the corresponding bypass eliminates both circles. By induction, we can then assume all inessential components of $R_\pm$ are disks. Since $\Sigma$ is a torus and the ambient contact structure $\xi|\Sigma=(\E/\W)|\Sigma$ is a trivial bundle, $\chi(R_\pm)=0$. Thus, if $R_+$ contains a disk then it also has a component with negative Euler characteristic, locating a disk in $R_-$. Since $\Gamma$ has only two essential components, these two disks are separated by a single essential circle in $\Gamma$. The two disks may then be eliminated by a bypass, inductively reducing to the case where $\Gamma$ consists only of two parallel essential circles. We can modify these by a Dehn twist parallel or perpendicular to $\Gamma$. (The former is the identity on $\Gamma$; the latter is Figure~\ref{DehnBypass}.) Since these Dehn twists generate $\SL(2;\Z)$, we can turn $\Gamma$ to the given direction. By the Flexibility Theorem with Addendum~\ref{C0}, we can isotope as allowed to obtain the required simple foliation.
\end{proof}

\begin{Remark}
The standard contact $\R^3$ contains an immersed overtwisted disk with only a single clasp double arc \cite[Proof of Proposition~5.1]{Ann}. We can then obtain an embedded overtwisted disk in a given even-contact 4-manifold by pushing along $\W$ as in the proof of Lemma~\ref{bypass}. It follows immediately that every surface $\Sigma$ without $\W$-tangencies lies in a $\W$-transverse 3-manifold $N$ for which $N-\Sigma$ is overtwisted. We can then control the characteristic foliation using Eliashberg's h-Principle \cite{E} and Gray's Theorem \cite{Gray}. This was the author's initial approach to Corollary~\ref{standardize}, but the bypass method is more explicit and can be made $C^0$-small. It follows from either approach that every neighborhood in an even-contact 4-manifold contains a $\W$-transverse overtwisted $\R^3$. (With bypasses, apply the Giroux Criterion.) In the Engel case, this embedding extends to a neighborhood of some section of the prolongation of the overtwisted $\R^3$ (cf.~Observation~\ref{tight}). The utility of this is unclear, since applications frequently require a $w$-interval exceeding some specific length, whereas the constructed neighborhood will typically be narrow in the $w$-direction.
\end{Remark}

\subsection{Transverse realization: proof of Theorem~\ref{main0}}\label{MainProof}

As we easily saw in Section~\ref{TransSurf}, every transverse surface $\Sigma$ in an Engel manifold $M$ is a torus with trivial normal bundle. Thus, Theorem~\ref{main0} is essentially the converse: We must make such a torus $\Sigma$ transverse by a $C^0$-small isotopy. Since $\E/\W$ is trivial on any Engel manifold, Corollary~\ref{standardize} makes $\Sigma$ convex with a simple foliation. The following theorem  then isotopes $\Sigma$ to be $\E$-transverse Legendrian.  Theorem~\ref{main0} follows immediately (for both orientations of $\Sigma$) by transverse pushoff (Proposition~\ref{push}).

\begin{thm}\label{main1}
In an Engel manifold $M$, suppose $\Sigma$ is a convex torus with simple characteristic foliation $\F$. Then $\Sigma$ is $C^0$-small isotopic in $M$ to an $\E$-transverse Legendrian $\Sigma^*$ with characteristic foliation given by the image of $\F$.
\end{thm}

\noindent Since the foliation of Corollary~\ref{standardize} is simple with any preassigned direction (Definition~\ref{stdForm}), and simple foliations are stable under small perturbations of the line field, the resulting foliation on the transverse pushoff $\tau\Sigma^*$ will be simple in any preassigned direction. For the proof of this theorem, we no longer explicitly need $\Sigma$ to be convex in its $\W$-transverse $N$, although it still follows from simplicity of $\F$ until the final stage, where circles of $\W$-tangencies are created so that $\Sigma^*$ no longer smoothly embeds in a $\W$-transverse $N$. These tangencies disappear under transverse pushoff, so $\tau\Sigma^*$ is again convex in $M$ since its foliation is simple. The proof of Theorem~\ref{main1} applies without change to any parametrized torus $\Sigma\approx\R^2/\Z^2$ without $\W$-tangencies and with the $x$-component of the characteristic line field positive everywhere. However, Reeb components would cause difficulties.

\begin{proof}
First, we need a way to measure the failure of $\Sigma$ to be Legendrian or transverse. We know that $\Sigma$ is $\E$-transverse since $\F$ is nonsingular. Equivalently, $\Sigma$ is $\xi$-transverse in the associated contact manifold $N$. However, $\Sigma$ has no clear relation to $\D$, since Corollary~\ref{standardize} provides no control over the line field $\L=\D/\W$ in $\xi$. We measure this failure by trivializing $\xi|\Sigma$ using the oriented line field $\xi\cap T\Sigma$. Then $\L$ determines a map $\psi_\L\co \Sigma\to S^1$. This is identically 0 or $\pi$ precisely when $\Sigma$ is Legendrian, and lies on an open arc of $S^1$ strictly between these if $\Sigma$ is transverse. The class $\psi_\L^*[S^1]$ in $H^1(\Sigma)$ is a delicate invariant, since it is only defined when the characteristic foliation is nonsingular. For example, it is unclear whether it has well-defined behavior under bypasses, since these require a global operation (the Flexibility Theorem) to restore a nonsingular foliation. However, we can control the invariant more directly. In fact, if $C\subset\Sigma$ is any circle transverse to $\F$, we can change $\psi_\L^*[S^1]$ by adding any multiple of $\PD[C]$, using a $C^0$-small isotopy of $\Sigma$ supported near $C$: We locally model $C\subset\Sigma\subset N$ as the $z$-axis in the $xz$-plane in $(\R^3,dz+xdy)$ mod unit $z$-translation. Then the Lagrangian projection of $\Sigma$ (Section~\ref{Proj}) is the $x$-axis in the $xy$-plane. We isotope this in $wxy$-space by a small Type I Reidemeister move over the $xy$-plane, toward either side of the $x$-axis. The corresponding $z$-invariant isotopy of $\Sigma$ in $M$ adds a vertical double curve to its projection in $xyz$-space. We draw the move in the $xy$-plane so that the resulting curve and the $x$-axis together enclose zero signed area. Then we can assume $\F$ is preserved since the new (Legendrian) leaves have no net change in $z$ across the modified region. Since the winding number in the $xy$-plane has been changed by $\pm 1$ but $\L$ has not changed, we have changed $\psi_\L^*[S^1]$, up to arbitrary sign, by $\PD[C]$. This can be done repeatedly, changing $\psi_\L^*[S^1]$ as desired. Note that $\Sigma$ is still embedded in a (different) $\W$-transverse $N$. By Definition~\ref{stdForm}, the simple foliation $\F$ comes with a parametrization $\Sigma\approx\R^2/\Z^2$ in which the slopes of the characteristic lines are bounded (by compactness). Thus, we can find a pair of transverse circles $C$ comprising a basis of $H_1(\Sigma)$, allowing us to change $\psi_\L^*[S^1]$ arbitrarily. We choose it to vanish. (Vanishing on a vertical circle is sufficient.)

We next wish to further simplify $\psi_\L$ by homotopy. Since $N$ is now likely to be tightly layered in $M$ (e.g.\ from the proof of Corollary~\ref{standardize}), we only have a small amount of maneuvering room in the $\W$-direction. Thus, we are only allowed small perturbations of $\L$ in $\xi$, so continue to rely on large changes of $T\Sigma$. We can at least perturb $\L$ so that $\psi_\L$ is generic. Then the set $\Lambda=\psi_\L^{-1}(\{0,\pi\})$ of $\L$-tangencies to $\Sigma$ is a 1-manifold in $\Sigma$ that has finitely many quadratic tangencies to $\F$. Any path in $\Sigma$ between components of $\Lambda$ has a well-defined $\L$-winding number in $\frac12\Z$ determined by $\psi_\L$. Suppose $C$ is an embedded arc in $\Sigma$ transverse to $\F$, intersecting $\Lambda$ exactly at its endpoints $\partial C$, where $\Lambda$ is tangent to $\F$ and convex on the sides away from $C$. If $C$ has $\L$-winding number zero then a $C^0$-small perturbation of $\Sigma$ in $N$, fixing $C$ but rotating its normal vectors slightly past $\L$, surgers $\Lambda$ along a band following $C$. This {\em transverse twist} eliminates the pair of tangencies without introducing others. Since this moves $\Sigma$ in $N$, we expect its characteristic foliation to change. However, we can model the transverse curve $C$ as for Proposition~\ref{helices} with $\Sigma$ a vertical plane.  Then twisting symmetrically preserves the foliation after a $C^1$-small isotopy of $\F$ in $\Sigma$. Thus, we recover the original hypotheses with fewer tangencies. We can also do a transverse twist if $\partial C$ contains a generic point of $\Lambda$ by first perturbing $\psi_\L$ to create a canceling pair of tangencies of $\Lambda$ with $\F$. The total number of tangencies will not increase unless we have done this at both endpoints of $C$. 

We can now simplify $\Lambda$ by transverse twists.  First we construct a component $\Lambda_0$ of $\Lambda$ that is a vertical circle (constant $x$ in $\R^2/\Z^2$). If $\Lambda$ is initially empty, we can just twist on a vertical circle (or note that $\Sigma$ is already transverse as desired, for one orientation). If $\Lambda$ has a nonvertical component, deform $\Lambda$ as in Figure~\ref{fingers} and consider an arc $C$ that runs the long way around a vertical circle as in the figure. Since $\Lambda$ is generic, it cuts $C$ into finitely many segments, each with $\L$-winding number $\pm\frac12$ or 0, and segments of opposite sign cannot be adjacent. Since $\psi_\L$ is nullhomotopic on vertical circles, $C$ has $\L$-winding number zero. Thus, some segment also has $\L$-winding number zero, so we can remove it by a transverse twist. By induction, we can now assume $\inter C$ is disjoint from $\Lambda$, and the final transverse twist splits off the required vertical circle $\Lambda_0$. Let $L$ be a leaf segment of $\F$ such that $L\cap\Lambda_0=\partial L$ and $L$ intersects $\Lambda$ transversely. Repeat the previous procedure at each intersection point of $\inter L$ with $\Lambda$ until all such intersections lie on vertical circles of $\Lambda$. After this, $\Lambda$ consists of a nonempty collection of vertical circles, together with some circles that are disjoint from $L\cup\Lambda_0$ and hence inessential in $\Sigma$. Each open annulus of $\Sigma$ bounded by two consecutive vertical circles can be reparametrized so that $\F$ is horizontal. In such a parametrization, we next arrange each circle of $\Lambda$ to have only one local maximum and minimum: Choose an innermost circle for which this fails. There must be a local extremum pointed inward, and we can find a transverse curve $C$ from this across the bounded region that is disjoint from any inner components of $\Lambda$ (each of which has a unique pair of local extrema by hypothesis). Then $C$ has $\L$-winding number 0 (since $C\cap\Lambda=\partial C$ lies in a single component of $\Lambda$), and its transverse twist splits the component in two without increasing the number of local extrema. This procedure increases the number of components until there is one for each pair of local extrema as required. Now for an outermost circle we can connect its two extrema by a transverse arc crossing $L$, converting it to a pair of essential circles without local extrema. By induction, we reduce to the case where $\Lambda$ consists entirely of parallel circles transverse to $\F$.

\begin{figure}
\labellist
\small\hair 2pt
\pinlabel {$\F$} at 168 68
\pinlabel {$\F$} at 8 68
\pinlabel {$\Lambda$} at 12 8
\pinlabel {$\Lambda$} at 172 8
\pinlabel {$C$} at 216 4
\pinlabel {$C$} at 216 55
\endlabellist
\centering
\includegraphics{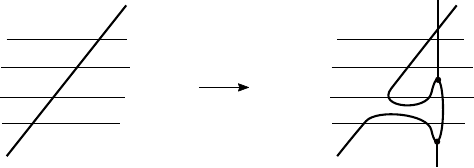}
\caption{Creating a vertical component of $\Lambda$.}
\label{fingers}
\end{figure}

To obtain the desired Legendrian torus in $M$, we perturb $\psi_\L$ to thicken $\Lambda$ to a collection of annuli with boundaries transverse to $\F$. Each annulus $A$ in between these has a neighborhood with a local model: Proposition~\ref{annuli} identifies a neighborhood of $A$ in $N$ with a neighborhood of the annulus given by $x\in[0,1]$, $y=0$ in $\R^2\times S^1=\R^3$ mod unit $z$-translation, with standard contact form $\alpha=dz+xdy$ so that $\F$ is horizontal. Then a neighborhood of $A$ in $M$ is identified with one in the prolongation, $S^1\times \R^2\times S^1$ with $\beta=\sin(w)dx+\cos(w)dy$. After a further perturbation of $\psi_\L$, we can assume $w|A$ is independent of $z$ near $\partial A$. For each fixed $z$, think of the corresponding leaf of $\F$ in $A$ as an arc in $wxy$-space, so $z$ parametrizes an $S^1$-family of arcs in the contact manifold $(S^1\times\R^2,\beta)$. Each is $\beta$-transverse (up to orientation) for $x\in(0,1)$ and $\beta$-Legendrian elsewhere. By Lemma~\ref{zigzag}, we can simultaneously isotope these arcs, supported with $x\in[0,1]$, to a family of $\beta$-Legendrian helices whose front projections are zig-zag arcs about the $x$-axis. The union of these $\beta$-Legendrian arcs over all $z$ is an embedded annulus $A^*$ in the 4-dimensional model that is $C^0$-small isotopic to $A$. Its projection into $\R^3$ mod $z$-translation has the appearence of a pleated curtain, intersecting the $xz$-plane in vertical curves, and failing to be smooth along creases composed of the cusps of the constant-$z$ zig-zag arcs. (The annulus $A^*$ is smooth in $M$ there, with the creases representing circles of $\W$-tangencies.) Since the isotopy creating the $\beta$-Legendrian arcs was $C^0$-small in the polar angle $-w$ of $\L$, the tangent planes to the smooth regions of the curtain can be assumed to be arbitrarily close to the vertical plane field $\ker\beta$ (which varies with $w$ on $A$ but is never tangent to $A$ for $x\in(0,1)$). The constant-$z$ $\beta$-Legendrian arcs on $A^*$ are no longer $\alpha$-Legendrian since their $y$-coordinates vary, so we replace them by the $\E$-characteristic foliation of $A^*$, which projects to the $\xi$-characteristic foliation of the curtain. Since the tangent planes of the curtain are close to the vertical planes $\ker\beta$, this change will be $C^1$-small on the $\alpha$-Lagrangian ($=\beta$-front) $xy$-projection of each arc (but only $C^0$-small on $z$ due to the slopes of $\xi$). A small perturbation of the curtain restores the condition that adjacent teeth have equal areas as in the proof of Lemma~\ref{zigzag}. Then each arc traverses $A^*$ with no net change in $z=-\int xdy$, so the $\E$-foliation agrees with $\F$ after a $C^0$-small isotopy of the latter in $A^*$. Since the new leaves have $\beta$-front projections $C^1$-close to the previous $\beta$-Legendrian (constant-$z$) arcs, they can be made $\beta$-Legendrian by a $C^0$-small isotopy of $A^*$ in the $w$-direction. This preserves the $\alpha$-Legendrian condition, so $\F$ is now tangent to $\D=\ker\alpha\cap\ker\beta$ everywhere on $A^*$. Applying this procedure to every $A$ in $\Sigma$ gives the required Legendrian torus $\Sigma^*$. This has no $\E$-tangencies since its projection to $N$ has no $\xi$-tangencies.
\end{proof}

\subsubsection{An operation on Legendrian tori} The above proof also yields the following operation that will be useful in Section~\ref{DT}:

\begin{sch}\label{Reeb}
Suppose $\Sigma$ is a Legendrian torus in an Engel manifold $M$, and its characteristic foliation has a nondegenerate closed leaf $L$ along which $\Sigma$ is transverse to $\E$ and has no $\W$-tangencies. Then there is a $C^0$-small isotopy of $\Sigma$ in $M$, supported near $L$, splitting $L$ into three parallel nondegenerate closed leaves separated by Reeb foliations as in Figure~\ref{ReebFig}. The modified region is still $\E$-transverse Legendrian, with no $\W$-tangencies on the closed leaves, but the Reeb regions have parallel circles of $\W$-tangencies.
\end{sch}

\begin{figure}
\labellist
\small\hair 2pt
\pinlabel {$L$} at 21 68
\endlabellist
\centering
\includegraphics{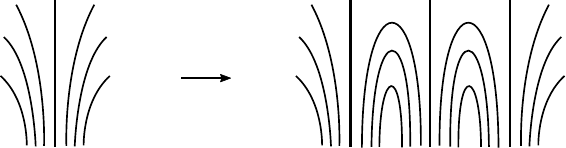}
\caption{Modifying the foliation of a Legendrian torus.}
\label{ReebFig}
\end{figure}

\begin{proof}
Since $\Sigma$ has no $\W$-tangencies on (hence near) $L$, the latter has a neighborhood $U\subset\Sigma$ lying in a contact 3-manifold $N$. After perturbing the foliation (e.g.\ by isotoping $\Sigma$ in $N$) we can assume the return map on $U$ is multiplication by a constant. We can then model $U$ in $(\R^3,dz+xdy)$ mod unit $y$-translation, with $L$ given by the $y$-axis and $U$ given by a generic plane containing it. The isotopy for the scholium is $y$-invariant, supported near the $y$-axis, fixing it but rotating a smaller neighborhood $V$ through a half-turn. Figure~\ref{fold} shows the projection into the $xz$-plane. The first step adds a new pair of parallel closed leaves (projecting into the $z$-axis in the figure) as we pass through the $yz$-plane. (This is {\em folding} as discussed in Section~\ref{Lerp}.) The remaining step reverses the direction of $L$ as in Example~\ref{liftReeb} when we pass through the $xy$-plane. (The foliation on $V$ is then the same as it was originally, but with reversed orientation since $V$ has half-turned.) Since $U$ was originally Legendrian, the line field $\L$ was tangent to the original foliation. If $U$ was sufficiently narrow, then $\L$ is nearly parallel to the $y$-axis. A small perturbation in the $\W$-direction then makes $\L$ tangent to our new foliation (up to orientation) everywhere except on a pair of annuli with closures lying in the open Reeb regions. The foliations on these annuli are transverse to $\L$, as Figure~\ref{ReebFig} shows. (In this figure, $\L$ lies in $\xi$ but is nearly vertical. The foliation rotates as we cross each annulus, exhibiting $\L$-winding numbers $\pm\frac12$ across each, with the same sign. The resulting full twist across the diagram corresponds to the obstruction exhibited in Example~\ref{liftReeb}.) Since the annuli can be reparametrized so that their leaves appear horizontal, we can identify them with the local model in the last paragraph of the previous proof, to make them $\E$-transverse Legendrian with the same foliation at the expense of introducing vertical circles of $\W$-tangencies.
\end{proof}

\begin{figure}
\labellist
\small\hair 2pt
\pinlabel {$x$} at 57 42
\pinlabel {$x$} at 179 42
\pinlabel {$x$} at 301 42
\pinlabel {$ z$} at 33 70
\pinlabel {$ z$} at 155 70
\pinlabel {$ z$} at 277 70
\endlabellist
\centering
\includegraphics{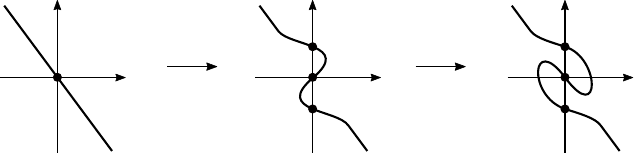}
\caption{Side view of the isotopy.}
\label{fold}
\end{figure}

\begin{Remark}\label{immersion}
If $\Sigma$ is the Legendrian lift of an embedded torus to a prolongation $\pr N$ (Example~\ref{liftReeb}), the foliation produced by the scholium is realized by the Legendrian lift of an isotopic torus in $N$. However, this lift will be in a different homotopy class from $\Sigma$ (as in Example~\ref{liftReeb}), whereas the scholium preserves the homotopy class. This is because the circles of $\W$-tangencies restore the original class. When we take transverse pushoffs, the Legendrian lifts become transverse lifts (Example~\ref{transLift}), determined by taking $\L$ to be the normal lines of the embeddings in $N$. However, some of the $\W$-tangencies of the scholium create circles of double points of the torus in $N$. (Compare with the transverse pushoff of a Legendrian knot, seen in a front projection.) The resulting extra winding of the normal lines $\L$ to the immersion restores the original homotopy class of $\Sigma$.
\end{Remark}


\section{Distinguishing isotopic transverse tori}\label{Distinguishing}

Having shown that a torus with trivial normal bundle in an Engel manifold is always isotopic to a transverse torus, we now investigate how many transverse tori can result. We consider two such tori to be equivalent if they are {\em transversely isotopic}, i.e., isotopic through transverse tori. (Such an isotopy need not extend to an ambient isotopy preserving the Engel structure; in particular, the leaf structure of the characteristic foliation may change.) More generally, we can ask when such tori are {\em transversely homotopic}, i.e., they are images of embeddings that are homotopic through transverse maps. Such a homotopy is automatically regular since transversality and a dimension count imply that a transverse map of a surface must be an immersion. We will study the {\em formal invariants} of transverse tori, invariants that, like the self-linking number of transverse knots in contact 3-manifolds, are determined by the underlying bundle theory. For this, we often suppress the geometry by working in a {\em formal Engel structure} on a 4-manifold $M$, an oriented flag of the form $\W\subset \D\subset\E\subset TM$, that we sometimes allow to vary with a parameter. Our formal invariants will be well-defined in this context, and preserved by isotopy (or sometimes homotopy) through maps transverse to the varying plane fields $\D$. After defining our main examples of such invariants in Section~\ref{Invariants}, we will discuss formal invariants more sytematically in Section~\ref{Classification}, culminating in the classification of formal transverse isotopy invariants. Section~\ref{Mod2} digresses on mod 2 residues. The remainder of the paper (Section~\ref{Examples}) then deals with the range and applications of these invariants in the Engel setting.

\subsection{The primary formal invariants}\label{Invariants} 

We begin with a pair of relative invariants. Given a formal Engel manifold $M$, let $f_0\co\Sigma\to M$ be an immersion transverse to $\D$. Then $f_0^*\D$ is the normal bundle of $f_0$ and $f_0^*(TM/\D)$ is identified with the tangent bundle $T\Sigma$. Both of these bundles are trivialized by the flag. A regular homotopy $F$ transports these normal and tangent trivializations to the final immersion $f_1$. If the latter is transverse, it also has trivializations induced by the flag, which are given by maps $\Sigma\to S^1$ relative to the transported trivializations. We denote the corresponding classes in $H^1(\Sigma)$ by $D_\nu(F)$ and $D_T(F)$. These are obviously invariants in the following sense:

\begin{prop}\label{transHom}
When $F$ is a transverse homotopy, $D_\nu(F)=D_T(F)=0$. \qed
\end{prop}

These invariants are only defined relative to a fixed regular homotopy $F$, so to obtain absolute invariants of transverse tori, we next introduce topologically determined reference framings.

\subsubsection{The invariant $\Delta_T$}
The tangent bundle $T\Sigma$ has a canonical homotopy class of framings: Simply identify $\Sigma$ with $\R^2/\Z^2$ and use the standard basis for $\R^2$. Changing the identification changes this description by a self-diffeomorphism of $\R^2/\Z^2$ which, after isotopy, is given by some $A\in\SL(2,\Z)$. The new framing is obtained from the old one by applying the same element $A$ to each tangent plane. This is homotopic to the original framing by fixing a path from $A$ to the identity in $\GL(2,\R)$ and applying it simultaneously to each tangent plane. Now for any transverse immersion $f\co\Sigma\to M$, define $\Delta_T(f)\in H^1(\Sigma)$ to be the class of the framing determined by $\E/\D$ in $f^*(TM/\D)\cong T\Sigma$, relative to the canonical framing. This determines the corresponding relative invariant as $$D_T(F)=\Delta_T(f_1)-\Delta_T(f_0)$$ for any regular homotopy $F$ between transverse immersions $f_0$ and $f_1$. Beware that there can be homologically nontrivial self-isotopies with $f_0(\Sigma)=f_1(\Sigma)$ but $D_T(F)\ne0$ (e.g.~Example~\ref{T3}). However, the divisibility of $\Delta_T$ is a transverse homotopy invariant of the image torus. We use the notation $\Delta_T=\Delta_T(\Sigma)$ when $f$ is an inclusion.

\subsubsection{The canonical normal framing}\label{Framing}
As is true for knots in 3-manifolds, we do not expect a canonical normal framing of a surface unless it is embedded, trivially in homology. In that case, it is the boundary of a {\em Seifert solid}, an embedded (compact, oriented) 3-manifold, whose outward normal to the surface is sometimes canonical:

\begin{prop}\label{Seifert}
Suppose $\Sigma$ is a surface embedded in a 4-manifold $M$, with $[\Sigma]=0$ in $H_2(M)$. Then its normal bundle $\nu\Sigma$ is canonically framed over any embedded circle $C$ in $\Sigma$ that has trivial intersection pairing with $H_3(M)$. If this is true for a basis of $H_1(\Sigma)$ then $\nu\Sigma$ is canonically framed.
\end{prop}

\noindent For example, the hypothesis on $C$ is true whenever $C$ vanishes in $H_1(M;\Q)$. When $[\Sigma]=0$, $\nu\Sigma$ is framed whenever the intersection pairing between $H_1(M)$ and $H_3(M)$ vanishes, such as if $M$ is $I\times N$ or the prolongation of any open 3-manifold with trivial pairing $H_1\otimes H_2$.

\begin{proof}
To verify that $\Sigma$ has a Seifert solid $N$, let $V$ be a closed tubular neighborhood of $\Sigma$ and consider the homomorphisms
$$[M-\inter V,S^1]\cong H^1(M-\inter V)\cong H_3(M-\inter V,\partial V)\cong H_3(M,\Sigma)\to H_2(\Sigma)\cong\Z$$
where we use compactly supported maps and cohomology if $M$ is noncompact. The arrow denotes the boundary map of the pair $(M,\Sigma)$, and is an epimorphism since $[\Sigma]=0$ in $H_2(M)$. The corresponding map onto $\Z$ is given by the intersection number with a meridian $\mu$ of $\Sigma$, so the entire composite epimorphism is given by restricting maps to $\mu$ and taking the degree. Thus, any map $M-\inter V\to  S^1$ that is sent to a generator restricts, after homotopy, to a diffeomorphism on each fiber of the circle bundle $\partial V\to\Sigma$. We can then take $N$ to be the preimage of any regular value, extended over $V$ in the obvious way.

The normal framing of $\Sigma=\partial N$ directed outward from $N$ is characterized along $C$ as the unique framing for which a pushoff $C'$ of $C$ in $M$ has $C'\cdot N=0$. To see that this is independent of choice of $N$ along $C$, let $N'$ be another Seifert solid, and let $\eta=[N-N']\in H_3(M)$. By hypothesis $C'\cdot \eta=0$, so $C'\cdot N'=0$ as required. The last sentence of the proposition follows since framings on $\nu \Sigma$ are classified by $H^1(\Sigma)$.
\end{proof}

\begin{Remark}
This fails whenever the hypothesis on $C$ does: If $C$ pairs nontrivially with some $\zeta\in H_3(M)$, we can add $\zeta$ to the class in $H_3(M,\Sigma)$  implicitly used to define $N$. The resulting $N'$ has $C'\cdot N'=C\cdot\zeta\ne 0$, so $N$ and $N'$ determine different framings along $C$. The statement without the hypothesis (which has appeared elsewhere in the literature) has a simple counterexample: Begin with a great circle bounding a disk $D$ in the round 3-sphere. Taking the product with $S^1$ yields a torus $\Sigma$ bounding a solid torus in $S^3\times S^1$. We can similarly obtain an infinite family of Seifert solids for $\Sigma$ by rotating $D$ in $S^3$ with fixed axis $\partial D$ as we traverse the $S^1$ factor. These Seifert solids determine all framings of $\nu\Sigma$ over any circle of the form $\{ p\}\times S^1\subset\Sigma$. In fact, these framings are all related by diffeomorphisms of $(S^3\times S^1,\Sigma)$, so none is canonically determined by the pair. The framing over $\partial D$ is canonical.
\end{Remark}

\subsubsection{The invariant $\Delta_\nu$}\label{DeltaNu}
Now for a transversely embedded torus $\Sigma$ in a formal Engel manifold $M$, with $[\Sigma]=0$ in $H_2(M)$, we choose a Seifert solid $N$ and define $\Delta_\nu=\Delta_\nu(\Sigma)$ in $H^1(\Sigma)$ to be the class of the framing induced by $\W|\Sigma$ in $\D|\Sigma= \nu\Sigma$, relative to the framing induced by $N$. Then $\Delta_\nu$ is independent of choice of $N$ whenever the latter framing is. To use this in full generality, let $A\subset H_1(\Sigma)$ be the summand consisting of all classes whose images in $H_1(M)$ have vanishing intersection pairing with $H_3(M)$. This annihilator is nontrivial, since it contains the nontrivial kernel of inclusion $H_1(\Sigma)=H_1(\partial N)\to H_1(N)$. We can now interpret $\Delta_\nu$ as a well-defined element of the dual space $A^*$, which is $H^1(\Sigma)$ when $\rk A=2$. When $\rk A\ne2$, we either interpret $\Delta_\nu$ as a well-defined element of $A^*\cong\Z$ or as a class in $H^1(\Sigma)$ depending on $N$. For a transverse embedding $f\co\Sigma\to M$, we similarly define $\Delta_\nu(f)$ using the cohomology of the domain. This invariant is analogous to the self-linking number of a transverse knot in a contact 3-manifold (Section~\ref{Knots}), and generalizes the self-linking class used by Kegel in \cite{K}. Unlike our previous invariants, $\Delta_\nu(f)$ is not generally preserved by transverse homotopy (e.g.\ Lemma~\ref{dnu}, also see below), but $N$ can be dragged along with a transverse isotopy. Thus:

\begin{cor} \label{wellDef}
When $[\Sigma]=0$, $\Delta_\nu(\Sigma)\in A^*$ is well-defined and preserved by transverse isotopies. Its divisibility in $\Z^{\ge0}$ is a transverse isotopy invariant that is independent of parametrization of $\Sigma$. \qed
\end{cor}

For any isotopy $F$ between transverse embeddings (but not for transverse homotopies, e.g.~Lemma~\ref{dnu}, Example~\ref{PV}) we have $$D_\nu(F)=\Delta_\nu(f_1)-\Delta_\nu(f_0).$$ (When $A\ne H_1(\Sigma)$, interpret the right side relative to a fixed $N$ and its isotopic image.) As with $D_T$, there can be self-isotopies for which this is nonzero (notably for an unknotted torus, Example~\ref{PV}(d)).

\subsubsection{A relation}\label{ds}
Our invariants must satisfy a mod 2 congruence. This arises since a formal Engel flag on $M$ trivializes $TM$, so determines a spin structure. The relative invariants for a regular homotopy $F$ between transverse immersions were obtained by transporting the flag framing from $f_0^*TM$ to $f_1^*TM$ using the induced splitting into plane bundles $\nu\Sigma\oplus T\Sigma$, then using it to measure the flag framing for the latter. Since $M$ is globally spin, the two induced spin structures on $f_1^*TM$ must be equal. Over each circle in $\Sigma$, $D_T(F)$ and $D_\nu(F)$ each count rotations of the flag framing of $f_1^*TM$ relative to the framing from $f_0$. Since the spin structures agree, the total number of rotations must be even. Thus, $$D_T(F)+D_\nu(F)\equiv 0 {\rm\ mod\ 2}.$$

For the absolute invariants, a Seifert solid $N$ for a nullhomologously embedded torus $\Sigma$ in any spin 4-manifold determines a class $\ds\in H^1(\Sigma;\Z_2)$ that compares the ambient spin structure to that determined by the canonical tangent and induced normal framings. This is isotopy invariant, and independent of $N$ to the same extent $\Delta_\nu$ is, although $A$ can be replaced by the larger $A_{\rm spin}\subset H_1(\Sigma)$ consisting of classes whose pairing with $H_3(M)$ is even. The class $\ds$ is often nonzero (Proposition~\ref{mod2}). In the formal Engel setting, $$\Delta_T+\Delta_\nu\equiv\ds {\rm\ mod\ 2},$$ relative to a given $N$ if $A\ne H_1(\Sigma)$.

\subsubsection{Computation in Engel manifolds}\label{Comp}
Recall (Section~\ref{EngelChar}) that since a transverse torus $\Sigma$ in an Engel manifold has no $\W$-tangencies, it projects into a canonical germ of a contact 3-manifold $N$, with the (nonsingular) $\E$-characteristic foliation projecting to the $\xi$-characteristic foliation. In particular, the notions of convexity and dividing sets are well-defined for transverse tori in Engel manifolds, and convexity can be arranged by a $C^\infty$-small perturbation transverse to $\W$. This viewpoint is helpful for explicitly computing the primary formal invariants:

\begin{prop}\label{comp}
\begin{itemize}
For $\Sigma$ as above,
\item[a)] The normal framing on $\Sigma$ determined by the Engel structure is induced by a vector field transverse to its image in $N$ and lying in $\xi$.

\item[b)] The invariant $\Delta_T\in H^1(\Sigma)$ vanishes on any embedded circle $C$ in $\Sigma$ that is transverse to the foliation or parallel to a closed leaf or dividing curve. If $\Sigma$ can be described as $\R^2/\Z^2$ so that the characteristic line field is never vertical, then $\Delta_T=0$. In particular, $\Delta_T=0$ for any transverse torus resulting from Theorem~\ref{main0}.

\item[c)] If $\Sigma$ is convex, then $\Delta_T$ is a multiple of the Poincar\' e dual of a dividing curve $\gamma$. Its signed multiplicity is half the number of Reeb components, counted with sign by their direction of convexity.
\end{itemize}
\end{prop}

For (c), note that nonsingularity of the foliation implies that all dividing curves are essential and hence parallel. By convexity, the foliation has finitely many closed leaves, all parallel to the dividing curves, so every Reeb component follows a dividing curve (cf.~text before Definition~\ref{stdForm}). A Reeb component contributes positively to the sum when its normal vectors on the convex sides of the leaves point in the direction of $\gamma$, or equivalently, when $\gamma$ is oriented oppositely to the direction of the flow for large and increasing $|t|$.

\begin{proof}
The framing in (a) is determined by a vector field in $\D$ transverse to $\W$, so that statement follows immediately. For the rest, we need to compare the Engel framing on $T\Sigma$ with the canonical reference framing. The former is given by $\E/\D$ on $TM/\D\cong T\Sigma$, which is the characteristic line field. The latter has degree 0 on any embedded essential circle in $\Sigma$, implying (b) (since the foliation produced by Theorem~\ref{main0} is simple by Theorem~\ref{main1} and surrounding text) and the first sentence of (c). For the multiplicity, observe how the characteristic line field rotates as we traverse a circle crossing each dividing curve once. To fix the sign (which we do not actually need), recall that $\langle\PD(\gamma),\zeta\rangle=\gamma\cdot\zeta$ for $\zeta\in H_1(\Sigma)$ (and beware that the intersection pairing is anticommutative).
\end{proof}

\subsection{The classification of formal transverse isotopy invariants}\label{Classification}

Given an immersion $f\co\Sigma\to M$ of a torus into a formal Engel manifold, both $\D$ and $T\Sigma$ determine subbundles of the pullback $f^*TM$ over $\Sigma$. In the language of the h-Principle, a {\em formal transverse immersion} (or {\em embedding}) is an immersion (embedding) together with a homotopy through bundle monomorphisms into $f^*TM$ that sends $f^*T\Sigma$ to a subbundle transverse to $f^*\D$. We take the equivalent perspective of homotopically sending $f^*\D$ to a subbundle transverse  to $f^*T\Sigma$, after which it can be identified as the pulled back normal bundle of the immersion. In the case of an embedding, such a homotopy extends to a homotopy of the formal Engel structure on $M$. We interpret any transverse immersion as formal via the constant homotopy. The invariants of the previous section immediately extend to the formal transverse setting. A {\em formal transverse homotopy (isotopy)} is a 1-parameter family of formal transverse immersions (embeddings). Clearly, $D_T(F)$, $D_\nu(F)$ and $\Delta_T$ are formal transverse homotopy invariants, and $\Delta_\nu$ is a formal transverse isotopy invariant. We wish to determine when an isotopy between formal transverse embeddings can be extended to a formal transverse isotopy. In our equivalent perspective, we allow the formal Engel structure to vary: Let $\D\subset I\times TM$ be the family of plane fields associated to a homotopy of formal Engel structures, and let $F\co I\times\Sigma\to I\times M$ denote an isotopy  between embeddings $f_0$ and $f_1$ transverse to the corresponding plane fields. Pulling back normal bundles gives a subbundle $\nu$ of $F^*TM$ agreeing with $F^*\D$ for $t=0,1$. We ask when $F^*\D$ is homotopic to $\nu$ rel $t=0,1$, or equivalently, when the formal Engel structures can be homotoped rel $t=0,1$ to make the embeddings given by $F$ transverse to the corresponding plane fields. The {\em formal transverse isotopy invariants} are the obstructions to making $F$ formally transverse in this manner. These consist precisely of $D_T(F)$, $D_\nu(F)$ and a pair of secondary obstructions that the author still finds somewhat mysterious:

\begin{thm}\label{formal}
The primary obstructions to making $F$ formally transverse are $D_T(F)$ and $D_\nu(F)$, which are congruent mod 2. When these vanish, $F$ can be made formally transverse if and only if a pair of $\Z$-valued secondary obstructions vanishes. For any transverse torus $\Sigma$ embedded in a formal Engel manifold, every possible combination of these obstructions is realized with $\Sigma$ fixed, by some homotopy of the formal Engel structure supported near $\Sigma$ after which $\Sigma$ is again transverse. In particular, such homotopies realize all classes in $H^1(\Sigma)$ as $\Delta_T(\Sigma)$. Given a Seifert solid, they realize all pairs of  classes with sum reducing mod 2 to $\ds$ as $(\Delta_T,\Delta_\nu)$.
\end{thm}

\begin{proof}
We interpret $F^*\D$ and $\nu$ as sections of the Grassmann bundle over $I\times\Sigma$ of oriented planes in the fibers of $F^*TM$. By hypothesis, these agree where $t=0,1$. The fiber is $\SO(4)/\SO(2)\times\SO(2)\approx S^2\times S^2$, where the two $\SO(2)$ factors correspond to rotations of the tangent and normal planes at the given point. Fix a cell decomposition of $\Sigma$ with only two 1-cells, and extend over $I\times\Sigma$ as the obvious product cell structure. Since $\pi_1(S^2\times S^2)=0$, we may assume $F^*\D=\nu$ over the vertical 1-cell (that extends the 0-cell of $\Sigma$). Since $\pi_2(\SO(4))=0$ and $\pi_1(\SO(4))=\Z_2$, the long exact homotopy sequence of the fibration injects $\pi_2(S^2\times S^2)$ onto the index-2 subgroup of $\pi_1(\SO(2)\times\SO(2))\cong\Z\oplus\Z$ consisting of pairs with even sum. By the definition of this boundary operator, the homotopy class of each vertical 2-cell of $I\times\Sigma$ maps to the element of the latter subgroup given by $(D_T(F),D_\nu(F))$ evaluated on the corresponding 1-cell of $\Sigma$. Thus, we can assume $F^*\D=\nu$ over both vertical 2-cells if and only if this pair vanishes. If so, only the 3-cell remains, corresponding to an element of $\pi_3(S^2\times S^2)\cong\Z\oplus\Z$.

To realize a given combination of obstructions over the specified $\Sigma$ (with $F$ given by inclusion) we first realize them abstractly by defining a section $\D'$ of the Grassmann bundle of $F^*TM=I\times (TM|\Sigma)$: Start by setting $\D'=\nu$ at $t=0,1$ and over the vertical 1-cell. The given $D_T$ and $D_\nu$ are required to have even sum, so for each vertical 2-cell there is a unique class in $\pi_2(S^2\times S^2)$ that we can use for extending $\D'$ relative to $\nu$. We can then extend over the 3-cell to realize the remaining obstruction in $\pi_3(S^2\times S^2)$. (Such extensions exist even if the primary obstructions do not vanish, since the attaching map of the 3-cell is nullhomologous rel $t=0,1$. However, in that case, the resulting invariant may only lie in a quotient torsor due to homotopies of the section over the 2-skeleton, cf.\ Remark~\ref{Theta}.) Since the new section $\D'$ agrees with $F^*\D$ when $t=0$, we can homotope the formal Engel structure on $M$ rel $t=0$ so that $F^*\D$ agrees with $\D'$ for all $t$. This realizes the given obstructions, and changes $\Delta_T$ and $\Delta_\nu$ arbitrarily as required.
\end{proof}

\begin{Remark}\label{Theta}
The secondary invariants are still defined (and realized as in the theorem) when the primary invariants are nonzero, although they may lie in finite cyclic torsors. To see this more clearly, trivialize $F^*TM$ and its Grassmann bundle so that $\nu$ is constant. Then $F^*\D$ can be interpreted as a pair of maps $I\times \Sigma\to S^2$ that are constant at $t=0,1$. But maps from a 3-complex to $S^2$ were classified by Pontryagin \cite{Po} (cf.~\cite[proof of Proposition~4.1]{Ann} for the Thom--Pontryagin approach). For each factor $S^2$, the primary obstruction is a class in $H^2(I\times\Sigma,\partial)\cong H^1(\Sigma)$ with some divisibility $d$, and the corresponding secondary obstruction ranges over $\Z/2d$ (as a torsor). The pair of primary obstructions carries the same information as the pair $(D_T,D_\nu)$, although the precise correspondence is not immediately clear.
\end{Remark}

\subsection{A digression on mod 2 residues}\label{Mod2}

When $\Delta_T+\Delta_\nu$ is well-defined in $H^1(\Sigma)$, its mod 2 residue $\ds$ (Section~\ref{ds}) is frequently nonzero. We provide details for completeness, although this is not needed elsewhere in the paper.

\begin{prop}\label{mod2}
Suppose an embedded torus $\Sigma$ in a spin 4-manifold $M$ vanishes in both $H_2(M;\Z)$ and $H_1(M;\Z_2)$. Then $\ds\in H^1(\Sigma;\Z_2)$ is nonzero.
\end{prop}

\begin{proof}
First note that $\ds$ is well-defined in $H^1(\Sigma;\Z_2)$, using the outward normal to any Seifert solid $N$, since the hypotheses guarantee that $N$ exists (Section~\ref{Framing}) and $A_{\rm spin}=H_1(\Sigma)$ (Sections~\ref{DeltaNu} and~\ref{ds}). We wish to evaluate $\ds$ via the quadratic form $q$ on $H_1(\Sigma;\Z_2)$ outlined by Rokhlin \cite{R} and described in more detail by Freedman and Kirby \cite{FK}. To define this, consider any embedded circle $C$ in $\Sigma$. Since $C$ vanishes in $H_1(M;\Z_2)$, it bounds a compact (not necessarily orientable) surface $F$ in $M$. By twisting this around its boundary if necessary, we can arrange the mod 2 intersection number $F\cdot\Sigma$ to vanish, and then add tubes to $F$ so that it intersects $\Sigma$ only along its boundary. We define $F\cdot F$ by intersecting $F$ with a transverse copy of itself, whose boundary is pushed off using a nowhere-zero vector field $\tau$ tangent to $\Sigma$ but normal to $C$. Then the quadratic form $q\co H_1(\Sigma;\Z_2)\to\Z_2$ is defined by setting $q[C]=F\cdot F$ for every such $C$.

To relate $\ds$ to $q$, note that when $C$ is essential in $\Sigma$, $\tau$ determines the canonical framing $\hat\tau$ of $T\Sigma$ along $C$. The framing $\hat n$ of $\nu\Sigma|C$ determined by the outward normal $n$ from $F$ differs from that induced by $N$ by an even number of twists. This is because the closure of $\inter F\cap N$ is a compact 1-manifold, whose boundary must be an even number of points in $C$. These points occur where $n$ agrees with the outward normal from $N$, so count the mod 2 degree of $n$ relative to $N$. It follows that the spin structure given by the framing $\hat\tau\oplus\hat n$ of $TM|C$ agrees with the one induced by $\Sigma$ and $N$. Since the given spin structure on $TM$ is obviously defined on $TM|F$, it agrees with $\hat\tau\oplus\hat n$ over $C$ if and only if the latter extends as a spin structure over $TM|F$. Now we can apply the Whitney sum formula for relative Stiefel--Whitney numbers:
$$\langle\ds,C\rangle =w_2(TM|F,\hat\tau\oplus\hat n)=w_2(\nu F,\hat\tau)+w_1(\nu F,\hat\tau)\smile w_1(TF,\hat n)+w_2(TF,\hat n)=$$ $$F\cdot F+w_1^2(F,C)+\chi(F)|_2=q[C]+1$$
for all essential $C$ in $\Sigma$. Note that both factors in the cup product can be identified with the obstruction $w_1(F,C)$ to orienting $F$ since $TM$ is orientable. The resulting square vanishes unless  $F$ is a M\"obius band summed with tori, which is precisely the case that $\chi(F)$ is even.

There are two isomorphism classes of quadratic forms on $H_1(\Sigma;\Z_2)$, distinguished by the Arf invariant. If $\Arf(q)$ vanishes, then $q$ is nonzero on a unique (nonzero) class in $H_1(\Sigma;\Z_2)$, and $\ds$ is Poincar\'e dual to it by the displayed formula. If $\Arf(q)$ were nonzero, then $q$ would be nonzero on all three nontrivial classes, implying $\ds=0$. To rule this out, let $K$ be a compact, codimension-0 submanifold of $M$ containing $N$ and surfaces $F$ for the three nontrivial classes. The double $DK$ is spin with signature $\sigma=0$.  This construction preserves $\Sigma$ and its Arf invariant, and its homology class is characteristic (in fact 0) in $H_2(DK;\Z_2)$. By \cite{R} or \cite{FK}, $\Arf(q)\equiv(\Sigma\cdot\Sigma -\sigma(DK))/8=0$ mod 2 as required.
\end{proof}

\begin{example}\label{R4mod2} Under the hypotheses of the previous proposition, we have characterized $\ds$ (and hence $\Delta_T+\Delta_\nu$ mod 2 in the formal Engel case) as Poincar\' e dual to the unique $\Z_2$-class of circles in $\Sigma$ bounding surfaces $F$ with $F\cdot F$ (mod 2) nonzero and $\inter F$ disjoint from $\Sigma$. For a simple example, every torus $\Sigma\subset S^3\subset\R\times S^3$ is the boundary of a tubular neighborhood of a knot. The 0-longitude and meridian bound oriented surfaces $F$ in $S^3$ with $\inter F\cap\Sigma=\emptyset$ but $F\cdot F=0$. Thus, $\ds$ is Poincar\'e dual to the remaining mod 2 class, their sum, which bounds a disk with odd $F\cdot F$ in the 4-manifold. Note that $\chi(F)$ is odd for each of these surfaces, so $n$ does not extend over $TF$. To get a nonorientable $F$, split the solid torus as a sum of two nonorientable line bundles. Each is a M\" obius band $F$ whose boundary wraps twice longitudinally around $\Sigma$, and an odd number of times meridionally, so it is mod 2 homologous to the meridian. Again we see that this class is not dual to $\ds$, but this time it is the cup product that is nonzero in the displayed computation, with each factor dual to the central circle. (A push in the fourth coordinate shows $F\cdot F$ still vanishes.)
\end{example}

\section{The range and applications of the formal invariants}\label{Examples}

We now investigate the range of combinations of formal invariants of transverse tori in Engel manifolds. We first show that in overtwisted Engel manifolds up to homotopy, all possible combinations are realized by any fixed torus, and in a sense, the invariants classify such tori. This is analogous to the behavior of self-linking numbers in overtwisted contact 3-manifolds, although the contact setting has the additional advantage that compactly supported homotopy through contact structures implies isotopy \cite{Gray}. In the hope of recognizing tight Engel structures (if these exist), we then study the range of invariants realized by isotopy classes in a fixed Engel manifold. We find that broad ranges of the invariants are often realized, but various gaps remain. These could potentially be useful for identifying tight Engel structures (Section~\ref{Future}). To compute the primary invariants, we often use the method of Section~\ref{Comp}, projecting to an embedded (or immersed) torus in a contact 3-manifold and applying Proposition~\ref{comp}. (The immersed case follows by pulling back to an embedding.)

\subsection{Overtwisted Engel structures}\label{Overtwisted}

This section examines the formal invariants of transverse tori in the setting of overtwisted Engel manifolds up to homotopy through such structures. Along with our classification of these invariants (Theorem~\ref{formal}),  we rely heavily on the h-Principle for overtwisted Engel structures developed in \cite{PV}. The latter shows that every formal Engel structure is homotopic to an overtwisted Engel structure, with their Theorem~1.1 presenting this in a suitable relative and parametric form. Overtwistedness is characterized by the presence of an {\em overtwisted 4-disk}, as defined in \cite{PV}. To include the situation of linked tori, we assume throughout the section that $\Sigma$ is an embedded union of tori in $M$ with trivial normal bundles. If $M$ is noncompact, we allow $\Sigma$ to have infinitely many components, but require its embedding and isotopies to be proper.

Recall from Section~\ref{Classification} that for the plane field $\D$ of a formal Engel structure on $M$, $\Sigma$ becomes a {\em formal transverse embedding} when paired with a fiber homotopy $h$ of $\D|\Sigma$ sending it through monomorphisms to a normal plane field $\nu\Sigma$. A transverse embedding corresponds to the case of constant $h$. A {\em formal transverse isotopy} is a 1-parameter family of formal transverse embeddings, where $\D$ sometimes represents a 1-parameter family of plane fields. A formal transverse embedding transforms by formal transverse isotopy under homotopy of $\D$ and isotopy of $\Sigma$.

We now show that in the overtwisted Engel setting up to homotopy, formal transverse embeddings can be made transverse. This allows complete flexibility in realizing the formal invariants  (Corollaries~\ref{exist1} and \ref{exist2}). Then we show how to make formal transverse isotopies transverse (Theorem~\ref{unique}), implying the formal invariants are a complete set of transverse isotopy obstructions in this setting (Corollary~\ref{unique1}).

\begin{thm}\label{exist}
Suppose $\D$ is an overtwisted Engel structure on $M$ and $h$ makes $\Sigma$ into a formal transverse embedding in $(M,\D)$. Then $h$ is formally transverse isotopic to a constant homotopy, via some homotopy of $\D$ through overtwisted Engel structures. The final structure is overtwisted on $M-\Sigma$, and the homotopy is supported in a preassigned connected neighborhood of $\Sigma$ union an overtwisted 4-disk.
\end{thm}

\noindent Applying \cite[Remark~1.2]{PV} in the proof shows that the homotopy can actually be supported in any neighborhood of $\Sigma$ whose components each contain an overtwisted 4-disk.

\begin{cor}\label{exist1}
If $\D$ is an overtwisted Engel structure on $M$ then it is homotopic, through other such structures, to one in which $\Sigma$ is transverse, with Engel framing realizing any preassigned tangent and normal framings on $\Sigma$ whose sum respects the spin structure on $M$.
\end{cor}

\noindent In particular, all classes in $H^1(\Sigma)$ can be realized as $\Delta_T$ (defined componentwise), and given Seifert solids (not necessarily disjoint) for some components of $\Sigma$, all pairs with sum reducing mod 2 to $\ds$ can be realized as their corresponding invariants $(\Delta_T, \Delta_\nu)$. Note that any formal Engel structure is homotopic to an overtwisted Engel structure, and hence, to one for which all of the above conclusions apply.

\begin{proof}
Since the sum of the given framings on $\Sigma$ respects the spin structure on $M$, it is homotopic to the Engel framing, so the theorem applies. (The homotopy exists over the 1-skeleton by the definition of spin structures, and extends over the 2-cells since $\pi_2(\SO(4))=0$.)
\end{proof}

\begin{cor}\label{exist2}
Suppose $\Sigma$ is transverse in an overtwisted Engel structure $\D$ on $M$.
\begin{itemize}
\item[a)] Then there is a homotopy of $\D$ through overtwisted Engel structures, after which $\Sigma$ is again transverse, that realizes any preassigned combination of formal obstructions for the constant isotopy of $\Sigma$.

\item[b)] Suppose $\D$ is overtwisted on $M-\Sigma$, and $F$ is an isotopy of $f_0=\id_\Sigma$ with $f_1(\Sigma)$ disjoint from $\Sigma$. Then after a homotopy of $\D$ through overtwisted Engel structures, rel a neighborhood of $\Sigma$, $f_1(\Sigma)$  is also transverse, with $F$ realizing any preassigned combination of formal obstructions.
\end{itemize}
\end{cor}

\noindent That is, in either case we can realize any $D_T$ and $D_\nu$ with even sum and any choices of the secondary obstructions (including the torsion invariants of Remark~\ref{Theta}), and do this (a) fixing $\Sigma$ or (b) in a single Engel structure homotopic to $\D$.

\begin{proof}
By Theorem~\ref{formal}, there is a homotopy $h$ making $\Sigma$ a formal transverse embedding realizing the desired formal obstructions. Then (a) immediately follows from Theorem~\ref{exist}. For (b), we instead use $F$ to transport $h$ by a formal transverse isotopy to a homotopy making $f_1(\Sigma)$ formally transverse. Theorem~\ref{exist} then makes $f_1(\Sigma)$ transverse, by a homotopy of $\D$ supported away from $\Sigma$. In the new Engel structure, this transverse embedding is formally transverse isotopic to $(\Sigma,h)$ along $F$, so $F$ realizes the given obstructions.
\end{proof}

In \cite{PV}, a parametrized family of Engel structures is called an {\em overtwisted} family if it contains a corresponding parametrized family of overtwisted 4-disks, which they call a {\em certificate of overtwistedness}. It follows from the proof of Theorem~\ref{exist} that the resulting homotopy of $\D$ in the theorem and corollaries is such an overtwisted family, and if we also assume $\D$ is overtwisted on $M-\Sigma$ then the resulting certificate avoids the given surfaces for all $t$. This notion is also useful in the following theorem, which discusses when an isotopy between transverse surfaces is a transverse isotopy, up to homotopy through overtwisted Engel structures. We more generally allow parametrized families $\W\subset\D\subset\E$ of formal Engel structures that are only Engel for some parameter values.

\begin{thm}\label{unique}
Suppose that $F\co I\times\Sigma\to I\times M$ is an isotopy of $f_0=\id_\Sigma$ and that $(\W,\D,\E)$ is a 1-parameter family of formal Engel structures on $M$. Suppose that for parameter values $i=0,1$, $(\W_i,\D_i,\E_i)$ is an Engel structure that is overtwisted on $M-f_i(\Sigma)$, and $f_i(\Sigma)$ is $\D_i$-transverse. If $F$ extends to a formal $\D$-transverse isotopy $h$ between these transverse embeddings, then $h$ generates a homotopy rel $t=0,1$ from $(\W,\D,\E)$ to an overtwisted family of Engel structures in which $F$ is a transverse isotopy. If $(\W,\D,\E)$ is originally an overtwisted family of Engel structures, with a certificate disjoint from the corresponding surfaces $f_t(\Sigma)$, then the resulting homotopy of $(\W,\D,\E)$ is an overtwisted 2-parameter family of Engel structures (with the corresponding disjointness of the certificate).
\end{thm}

The resulting corollary gives a sense in which Theorem~\ref{formal} supplies a complete set of obstructions to an isotopy being transverse in the overtwisted setting. For a fixed Engel structure $\D$ on $M$, suppose $F$ is an isotopy between transverse embeddings. We will call $F$ a transverse isotopy {\em up to homotopy through Engel structures} if the constant homotopy of $\D$ is homotopic rel $t=0,1$, through Engel structures, to a 1-parameter family for which $F$ is transverse.

\begin{cor}\label{unique1}
For a given Engel structure $\D$ and isotopy $F$ between transverse embeddings, suppose there is an overtwisted 4-disk in $M$ disjoint from the image of $F$. Then $F$ is a transverse isotopy up to homotopy through Engel structures if and only if the obstructions of Theorem~\ref{formal}  vanish.
\end{cor}

\begin{proof}
By Theorem~\ref{formal}, $F$ extends to a formal transverse isotopy if and only if the obstructions vanish. If so, apply Theorem~\ref{unique} to the constant family $\D$ of Engel structures. Conversely, the hypothesized homotopy rel $t=0,1$ determines the required formal transverse isotopy.
\end{proof}

\begin{proof}[Proof of Theorem~\ref{exist}]
It suffices to assume $\D$ is overtwisted on $M-\Sigma$: By hypothesis, there is an overtwisted 4-disk in $M$. This is isotopic to a 4-ball $B$ disjoint from $\Sigma$. Dragging the plane field $\D$ along this isotopy can be interpreted as homotoping $\D$ through overtwisted Engel structures, after which $B$ is the required disjoint overtwisted 4-disk.

Next we arrange $\D$ to agree with a suitable local model near each component $\sigma$ of $\Sigma$, by temporarily sacrificing the Engel condition. The homotopy $h$ transfers the Engel framing to a tangent and normal framing of $\sigma$. Let $d\ge0$ be the divisibility of the corresponding class $\Delta_T(\sigma,h)\in H^1(\sigma)$. Any convex torus in a tight contact 3-manifold can be converted, by folding (Section~\ref{Lerp}) and the Flexibility Theorem~\ref{flex}, to one with $2d$ Reeb components with the same sign. Its transverse lift (Example~\ref{transLift}) in the prolongation then has $\Delta_T$ with divisibility $d$ (Proposition~\ref{comp}(c)). A neighborhood of this lift can be diffeomorphically mapped to a neighborhood of $\sigma$, sending $\Delta_T$ to $\Delta_T(\sigma,h)$. We can also control the normal directions so that $\W$ in the model maps to a line bundle determining the normal framing of $\sigma$ induced by $h$. Then $h$ generates a homotopy of formal Engel structures, ending with the model Engel structure near $\Sigma$.

The theorem now follows easily from \cite{PV}: The constructed formal Engel structure is Engel near $\Sigma$ and overtwisted Engel far from $\Sigma$, but only formally Engel on some intermediate region. However, \cite[Theorem~1.1]{PV} (with one-point parameter space $K$) generates a homotopy, supported in a given connected neighborhood of $\Sigma\cup B$, rel $B$ and a smaller neighborhood of $\Sigma$, yielding an Engel structure. The same theorem, with parameter space now an interval, adjusts the homotopy rel $t=0,1$, to be through Engel structures.
\end{proof}

\begin{proof}[Proof of Theorem~\ref{unique}]
We use a 1-parameter version of the previous proof. The given formal transverse isotopy transfers the family of Engel framings to tangent and normal framings on the embeddings $f_t(\sigma)$. We need to find a smooth family of local models respecting these framings and interpolating between the given Engel structures at $t=0,1$. Then we can homotope $(\W,\D,\E)$ to agree with these models, yielding the required transverse isotopy locally. By hypothesis, there is a $\D_i$-overtwisted 4-disk in $M-f_i(\Sigma)$ for $i=0,1$. Since these are isotopic in $M$ avoiding the surfaces $f_t(\Sigma)$ for each $t$, we can interpolate to get a certificate for $(\W,\D,\E)$ as in the proof of \cite[Corollary~1.4]{PV}. Their Theorem~1.1, with a 1-dimensional parameter space, implies our first conclusion (after a level-preserving self-diffeomorphism of $I\times M$ so that $F$ and the certificate temporarily appear constant). The remaining conclusion then follows similarly with a 2-dimensional parameter space. (In that case, we can also control the support of the homotopies as in the preceding proof.)

To construct the required smooth family of local models, we must first find suitable local models for each $f_i(\sigma)$ in $(M,\D_i)$, $i=0,1$, up to perturbation of $f_i$ and $\D_i$. Each $f_i(\sigma)$ canonically projects to a torus $\sigma_i$ with nonsingular characteristic foliation in a germ of a contact manifold (Section~\ref{EngelChar}). We can assume $\sigma_i$ is convex after a $C^\infty$-small perturbation. Since the foliation is nonsingular, the dividing set consists of a nonzero number of parallel essential curves. Then we have the same (finite) number of closed leaves, all parallel. (These are interleaved with the dividing curves since the annulus between any two closed leaves is convex with Legendrian boundary, so has nonempty dividing set.) However, we have no information about the number and signs of Reeb components, except that their signed count is determined by $\Delta_T$ (Proposition~\ref{comp}(c)). After a further perturbation, we can assume each return map is multiplication by a constant, so the foliation is modeled at each closed leaf by the $y$-axis in a diagonal plane in the standard contact $\R^3$ mod unit $y$-translation. For a model including all of $\sigma_i$, we instead work in the standard tight contact 3-torus $T^3$ given by $(\R^3/\Z^3,\sin(2\pi x)dy+\cos(2\pi x)dz)$. We construct a model that is $y$-invariant, and whose projection to the $yz$-coordinate torus is an orientation-preserving diffeomorphism. We choose an oriented circle $C_i$ in $\sigma_i$ intersecting each closed leaf once, and require its image in the model to have strictly increasing $z$-coordinate. Each closed leaf is then modeled in $T^3$ by setting $x$ equal to either 0 or $\frac12$ in some $y$-invariant annulus, with the choice of $x$ determined by the sign of intersection with $C_i$. (Note that the leaves are oriented oppositely for the two $x$-values since $\xi$ is horizontal at both values, but with opposite orientations.) We can now embed all of $\sigma_i$ in $T^3$, respecting the foliation: Follow $C_i$ around $\sigma_i$, embedding each closed leaf as it is encountered, and $y$-invariantly filling in the annuli between these leaves so that $C_i$ projects diffeomorphically to the $z$-coordinate circle. Reeb components then correspond to annuli spanning the two values of $x$. Thus, the $x$-winding number of the resulting torus $\sigma_i^*$ in $T^3$ is determined by their signed count, and hence by $\Delta_T(f_i)$ (Proposition~\ref{comp}(c)). Since the embedding respects the foliations, it extends contactomorphically to neighborhoods in the 3-manifolds (Section~\ref{SurfCont}) and then to a model for $f_i(\sigma)\subset M$ (up to the above perturbations) as a transverse lift in the prolongation of $T^3$ (Section~\ref{LocalEngel}).

Now we can find a smooth family of local models interpolating between $f_0(\sigma)$ and $f_1(\sigma)$. (We do this abstractly, postponing the embeddings in $M$ until the next paragraph.) First note that any isotopy of $\sigma_i^*$ in $T^3$ through $\xi$-transverse tori smoothly varies the local model initially displaying $f_i(\sigma)$. We guarantee $\xi$-transversality by varying through tori projecting diffeomorphically to the $yz$-coordinate torus. We also retain $y$-invariance. The given formal transverse isotopy guarantees that $\Delta_T(f_0)=\Delta_T(f_1)$. If these vanish, the corresponding tori $\sigma_i^*$ have vanishing $x$-winding number in $T^3$, so are isotopic to tori with constant $x$. We do not know the identification of $\sigma_1^*$ with $\sigma$ (which was constrained by the unknown direction of the dividing curves of $\sigma_1$). However, the foliations on the constant-$x$ tori have constant slope, and any slope can be realized by suitably choosing $x$. Thus, we can assume the diffeomorphism $\sigma_0^*\approx\sigma_1^*$ determined by their identifications with $\sigma$ preserves the foliation. Then the models agree, by uniqueness of the contact neighborhood and transverse lift (Example~\ref{transLift}). We can now fit the two smooth families together to interpolate as required. If $\Delta_T(f_i)$ is not zero, its direction determines that of the (unoriented) dividing curves of each $\sigma_i$ (Proposition~\ref{comp}(c)), which is horizontal in $\sigma^*_i$. If we choose the circles $C_i\subset \sigma_i$ to be images of the same oriented circle in $\sigma$, then the Reeb components of the two tori $\sigma^*_i$ have the same signed count relative to the positive $y$-direction, so these tori have the same $x$-winding number. Thus, we can isotope $\sigma_1^*$ onto $\sigma_0^*$. Their induced diffeomorphisms with $\sigma$ may still differ by horizontal Dehn twists. But we can also assume $\sigma_0^*$ has an annulus lying in the torus $x=0$, which is then foliated by circles. Since this part of the foliation is invariant under the allowed Dehn twists, we can again construct an interpolating family.

Finally, we must embed the local models into $M$ so that the 1-parameter family of Engel framings on the transverse tori $f_t(\sigma)$ in the models agrees with the family induced by the given formal transverse isotopy, up to homotopy rel $t=0,1$ preserving transversality of each $f_t(\sigma)$. These framings agree as required for each $t$ since they both come from the Engel framing on $\sigma$. However, they may differ on the family by some number of global rotations of all fibers of either $T\sigma$ or $\nu\sigma$ as $t$ increases. (This ambiguity, like the invariants $\Delta_T$ and $\Delta_\nu$, arises from $\pi_1(\SO(2)\times\SO(2))$ as we lift from the Grassmannian to $\SO(4)$.) But note that we had a choice in how to isotope $\sigma_1^*$ to its final location: Since $\xi$ undergoes a full rotation as we increase $x$ by 1 in $T^3$, translating $\sigma_1^*$ once around $T^3$ induces a full rotation of its characteristic line field. Thus, we can choose the family of models so that the framings suitably agree on $T\sigma$. We can then embed the models so that the normal framings also correspond.
\end{proof}

\subsection{Varying $\Delta_\nu$ with $\Delta_T=0$}\label{Dnu}

It remains to consider examples of a given torus $\Sigma$ in a 4-manifold $M$ with a fixed Engel structure. We investigate the range of primary formal invariants that can be realized by an isotopy of $\Sigma$ (usually $C^0$-small). As we saw in Proposition~\ref{comp}(b), it is common for transverse tori to have $\Delta_T=0$. Notably, the tori produced from our main existence Theorem~\ref{main0} always have this property. We now investigate the range of $\Delta_\nu$ under isotopies yielding vanishing $\Delta_T$. When $\Sigma$ is initially transverse, our resulting tori are simultaneously obtained both by isotopy and transverse homotopy, which can be simultaneously chosen $C^0$-small. (Examples that cannot be transversely homotopic, distinguished by $\Delta_T$, appear in the next section.) The main lemma of this section varies $D_\nu$ (hence $\Delta_\nu$) by adapting the stabilization operation for transverse knots in contact 3-manifolds. Note that for any torus $\Sigma$ in an Engel manifold, if a curve $C\subset\Sigma$ is transverse to the characteristic foliation (so avoids any singularities) then it is canonically oriented by requiring $\alpha|C$ to be positive. Recall that $\PD$ denotes Poincar\'e duality.

\begin{lem}\label{dnu}
Suppose $C\subset\Sigma\subset M$ is a circle (positively) transverse to the characteristic foliation of a transverse torus in an Engel manifold. Choose an integer $n\ge0$. Then there is an isotopy $F$ and a transverse homotopy $H$, both arbitrarily $C^0$-small, from the inclusion of $\Sigma$ to a transverse embedding $f_1=h_1$, such that $D_T(F)=D_T(H)=D_\nu(H)=0$ and $D_\nu(F)=2n\PD[C]\in H^1(\Sigma)$.
\end{lem}

\begin{proof}
Note that $C$ is primitive since it is embedded and homologically essential (for otherwise the foliation would have a singularity). After canonically projecting $\Sigma$ into its associated contact 3-manifold, we have a local model of $C$ that identifies it with the positively oriented $z$-axis in $(\R^3,dz+xdy)$ mod unit $z$-translation, and sends its neighborhood in $\Sigma$ into the $yz$-plane (e.g.\ by Proposition~\ref{annuli}). Then the characteristic foliation is given by lines of constant $z$. A neighborhood of $C$ in $M$ is now identified with one in the prolongation of the model via the induced line field $\L$ as in Section~\ref{LocalEngel}. Since $\Sigma$ is transverse, $\L$ is positively transverse to the plane, and it is cut out from $\xi$ by the 1-form $\beta=\sin(w) dx +\cos(w) dy$ of Section~\ref{Prolong}. The characteristic foliation is oriented (Section~\ref{OrientL}) so that $\beta$ is positive on it. (If we choose the model so that $(y,z)$ gives the canonical orientation of $\Sigma$, then the foliation is parallel to the positively oriented $y$-axis, and $\L$ is oriented toward increasing $x$.) Thus, in each slice of constant $z$, the corresponding leaf is exhibited as a positively $\beta$-transverse curve in the $wy$-plane in $wxy$-space. We can stabilize this as in Figure~\ref{stab} (that projects out the $w$-axis). Stabilizing a transverse knot in $\R^3$ lowers its self-linking by 2, by raising the 0-framing by 2 relative to the framing induced by the contact structure (Section~\ref{KnotOps}). In our case, we do not have a canonical 0-framing, but the same change occurs to any preassigned local framing. Stabilizing $n$ times, varying smoothly with respect to $z$ as $\ker\beta$ varies through transverse line fields in the diagram,  creates a new surface that is both isotopic and transversely homotopic to $\Sigma$. (Each resulting immersion in the homotopy is transverse since its new characteristic line field, which is nearly horizontal in the $xyz$-space model, determines a tangent framing on which $\beta\wedge\alpha$ is positive.) The homotopy (denoted $H$) carries along the original Engel normal framing, and the isotopy ($F$) carries it topologically but not through Engel framings. The resulting mismatch shows that $\D_\nu(F)$ has value $-2n$ on a circle $C'$ intersecting $C$ once with $C'\cdot C=1$. (Since $C$ is primitive, $C'$ exists. The order  of the factors is chosen so that $C'$ crosses $C$ in the same direction as the foliation, cf.~Section~\ref{OrientL}.) The new embedding $f_1$ only changes the foliation by a small perturbation near $C$. Thus, $D_T(F)$ and $\langle D_\nu(F),C\rangle$ vanish and $\langle D_\nu(F),C'\rangle=-2nC'\cdot C=+2nC\cdot C'$, so $D_\nu(F)=2n\PD[C]$. The other invariants listed in the lemma must vanish since $H$ is a transverse homotopy (Proposition~\ref{transHom}).
\end{proof}

\begin{cor}\label{Z}
If the foliation on a transverse $\Sigma$ has a closed leaf $L$, then $\Sigma$ can be changed as above so that $D_\nu(F)=2n\PD[L]$ for any integer $n$.
\end{cor}

\begin{proof}
After a small perturbation, we can assume $L$ is nondegenerate. A parallel pushoff $C$ of $L$ is then transverse to the foliation, and its orientation depends on which side it has been pushed toward.
\end{proof}

\begin{cor}\label{halfplane}
If the foliation on a transverse $\Sigma$ has both a closed leaf $L$ and a transverse arc $A$ intersecting $L$ in two points, then there are homotopies $F$ and $H$ as in Lemma~\ref{dnu} realizing all even classes $\eta\in H^1(\Sigma)$ with $\langle\eta,L\rangle\le0$ as $D_\nu(F)$ with $\Delta_T(f_1)=0$.
\end{cor}

Dually, we can equivalently realize all classes $2\zeta\in H_1(\Sigma)$ with $L\cdot\zeta\ge0$.

\begin{proof}
As before, assume $L$ is nondegenerate. Every primitive class $\zeta\in H_1(\Sigma)$ with $L\cdot\zeta\ge0$ can be realized by a positively transverse circle $C$ in $\Sigma$, made by connecting parallel copies of $A$ to arcs of $L$ pushed off to the side guaranteeing positivity as in the previous corollary. Lemma~\ref{dnu} now realizes all required classes. Since the original foliation agrees with the canonical trivialization of $T\Sigma$ along the subset $L\cup A$ carrying $H_1(\Sigma)$, $\Delta_T(\Sigma)=0=\Delta_T(f_1)$ (the latter since $D_T(F)=0$).
\end{proof}

\begin{thm}\label{cofinite}
For any torus $\Sigma$ embedded with trivial normal bundle in an Engel manifold $M$, there is a family $F^i$, for $i\in\Z^+$, of $C^0$-small isotopies from the inclusion to transverse embeddings $f^i$, such that the invariants $D_\nu(\hat F^i)$, where $\hat F^i$ is the induced isotopy from $f^1$ to $f^i$, range over all but finitely many elements of $2H^1(\Sigma)$. The transverse embeddings $f^i$ have $\Delta_T=0$ and comprise no more than three transverse homotopy classes.
\end{thm}

\begin{proof}
By Theorem~\ref{main1} and surrounding text, $\Sigma$ is $C^0$-small isotopic to a transverse torus with simple characteristic foliation, and with a closed leaf realizing any preassigned primitive homology class. This satisfies the hypotheses of Corollary~\ref{halfplane}. Define $F^1$, $F^2$ and $F^3$ by choosing three homology classes so that applying Corollary~\ref{halfplane} to each fills the complement of a triangular region. (We can do this with no information about $D_\nu(\hat F^2)$ or $D_\nu(\hat F^3)$, which translate the half-planes an unknown amount relative to $f^1$.)
\end{proof}

We can now prove that every torus $\Sigma$ embedded with trivial normal bundle in an Engel manifold is isotopic to infinitely many transverse tori as in Theorem~\ref{many}. Some care is required since there are such tori for which $\Delta_\nu$ is well-defined in $H^1(\Sigma)$ that have homologically nontrivial self-isotopies with nonzero $D_\nu$ (Example~\ref{PV}(d)). However, the divisibility of $\Delta_\nu$ is an isotopy invariant.

\begin{proof}[Proof of Theorem~\ref{many}]
For Part (b), we are given that $[\Sigma]=0$. Theorem~\ref{cofinite} realizes values $\Delta_\nu(f^i)=\D_\nu(\hat F^i)+\Delta_\nu(f^1)$ with arbitrarily large divisibilities in $A^*$. (These can be chosen in the same half-plane so that the embeddings are transversely homotopic.) By Corollary~\ref{wellDef}, the resulting images are all distinct under transverse isotopy and reparametrization. For Part (a), we allow $[\Sigma]$ to be nonzero, so that $\Delta_\nu$ may no longer be defined. Isotope $\Sigma$ to eliminate any $\W$-tangencies (Proposition~\ref{noW}). Then $\Sigma$ canonically embeds in a germ of a contact manifold diffeomorphic to $\Sigma\times\R$. This smoothly embeds in $\R^3$, and its contact structure extends (typically overtwisted) over $\R^3$ by Eliashberg~\cite{E}. Then a neighborhood $U$ of $\Sigma\subset M$ embeds in the corresponding prolongation $P\approx S^1\times\R^3$, preserving the Engel structure. But $H_2(P)=0$, so the previous case ($C^0$-small) shows we can find the required tori in $U$. These cannot be related by transverse isotopies in $P$, so also not in $U$.
\end{proof}

We conclude the section with explicit computations for a family of examples.

\begin{example}\label{PV}
a) Kegel \cite{K} distinguishes infinite families of transverse tori that are isotopic but not transversely isotopic, using an invariant determined by $\Delta_\nu$. These families were first constructed by del~Pino and Vogel \cite{PV} from an $\E$-transverse {\em core} circle $C$ in an Engel manifold $M$ and a transverse knot ({\em profile}) $K$ in $\R^3$. To describe these results from our perspective, locally model the core $C$ by the $z$-axis in $N=(\R^3,dz+xdy)$ mod unit $z$ translation, lifted to the prolongation $\pr N$ so that $C$ projects to 0 in $wxy$-space. (This can be arranged by thickening $C$ to a contact 3-manifold $N_C$ in $M$, then applying Proposition~\ref{annuli} to an annulus in $N_C$ obtained by flowing $C$ along $\L$, so that $\L|C$ maps parallel to the positively oriented $x$-axis in $N$.) The profile $K$ can be exhibited as a transverse knot in a small neighborhood of 0 in $wxy$-space with contact form $\beta=\sin(w) dx+\cos(w) dy$ from Section~\ref{Prolong}. Let $\Sigma$ be its preimage in the model from projecting out $z$. If we instead project out $w$, its image $\p(\Sigma)$  in $N$ is a $z$-invariant immersed torus whose constant-$z$ cross sections are the front projection of $K$. Since homotopic $\E$-transverse circles $C$ in $M$ are transversely isotopic \cite[Lemma~2.10]{PV}, the dependence of $\Sigma$ on $C$ is limited. However, when $M$ is simply connected, $M-\Sigma$ has the same fundamental group as $\R^3-K$, so varying the knot type of $K$ yields many isotopy classes. Each constant-$z$ slice of $\Sigma$ is $\beta$-transverse by construction, and the tangent vectors parallel to the $z$-axis lie in $\ker\beta$, so we have bases everywhere on which $\beta\wedge\alpha$ is positive, implying $\Sigma$ is transverse (since $\D=\ker\alpha\cap\ker\beta$).  Since the characteristic foliation on $\Sigma$ is never parallel to the $z$-axis, $\Delta_T=0$ (Proposition~\ref{comp}(b)). To compute $\Delta_\nu$, fix a positive basis $(\mu,\lambda)$ for $H_1(\Sigma)$, where $\mu$ is a constant-$z$ copy of $K$ (positively oriented by $\beta$) and $\lambda$ is parallel to the positively oriented $z$-axis. A Seifert surface for $K$ in $\R^3$ pulls back to a $z$-invariant Seifert solid for $\Sigma$. Its outward normal in $M$ agrees along $\lambda$ with the normal to $\p(\Sigma)$ in $N$, so $\langle\Delta_\nu,\lambda\rangle=0$ (Proposition~\ref{comp}(a)). Along $\mu$, the Engel normal framing is the blackboard framing of $K$ (given by $\W$ in $M$), which we measure relative to the Seifert framing. Thus, $\langle\Delta_\nu,\mu\rangle$ is the self-linking number $l(K)$ (Section~\ref{Knots}), and
$$\Delta_\nu=-l(K)\PD(\lambda)$$
(since $\lambda\cdot\mu=-1$). From Section~\ref{DeltaNu}, $\Delta_\nu$ is well-defined in $H^1(\Sigma)$ when $C$ pairs trivially with $H_3(M)$, and on $A=\langle\mu\rangle$ otherwise. Its divisibility $|l(K)|$ then distinguishes infinitely many transverse isotopy classes within the isotopy class determined by a given topological knot type of $K$ and fixed $C$. Varying the knot type of $K$, these are all transversely homotopic for a given $C$ (by the corresponding statement for transverse knots in $\R^3$). Note that changing $\Sigma$ by stabilizing $K$ is a special case of the procedure of Lemma~\ref{dnu}.

\item[b)]  We can now explicitly find even more classes realized as $\Delta_\nu$ of tori isotopic and transversely homotopic to a given $\Sigma$ as above. We arrange our transverse knot $K$ so that the total signed area enclosed by its front projection vanishes. (At the two points of extremal $x$ in the above front projection, the tangent vectors cannot point directly downward, so they point upward and contribute area with opposite sign. After rescaling $K$ sufficiently small in its required neighborhood, we can push one of these points outward to cancel the rest of the area.) Then the characteristic foliation of $\Sigma$ consists of circles representing $\mu$ in homology (since projecting $z$ out of $N$ is $\alpha$-Lagrangian, Section~\ref{Proj}). Applying Corollary~\ref{halfplane} realizes all classes of the form
$$\Delta_\nu=2m\PD(\mu)+(2n-l(K))\PD(\lambda)$$ with $m,n\in\Z$ and $n\ge0$.
The inequality can be weakened to $n\ge-\epsilon|m|$ for some small positive $\epsilon$ by perturbing $K$ to have front projection with nonzero signed area of either sign. (Then we can arrange closed leaves with any rational slope close to that of $\mu$.) The method of Theorem~\ref{cofinite} realizes all but finitely many classes congruent to $\PD(\lambda)$ mod 2 by adding another half-plane. However, we have no obvious bound on the number of classes missed, and it isn't clear if the tori realizing the new half-plane are transversely homotopic to $\Sigma$.

\item[c)] Now we specialize to the case $M=M_r\approx\R^4$ of Observation~\ref{tight}. Then in the smooth category, $C$ is unknotted and has only two possible framings in $\R^4$. The knotted tori we realize are precisely the turned spun tori. (The framing is ``turned'' since $l(C)$ is always odd.) Every $M_r$ contains a neighborhood in $\pr(\R^3,dz+xdy)$ of the section $w=0$, and $C$ is vertically transversely isotopic in $M_r$ to a transverse knot in this $\R^3$ section. We can draw this knot by its front projection in the $yz$-plane. The untwisted framing of $C$ in our 3-dimensional model $N$ in (a) (given by both $\L|C$ and $\lambda$) becomes its blackboard framing (parallel to the $x$-axis) when viewed in the given section of $M_r$. Then $\Sigma$ is pictured as an immersed torus following $C$ using this framing, with each slice normal to $C$ given by the front projection of $K$. The above computations apply with $\lambda$ the blackboard longitude, which is canonical in this setting: We can change any crossing of $C$ with at least one upward strand (in ``x'' position) by a transverse isotopy bypassing the crossing in the fourth coordinate. This changes $l(C)$ and the 0-longitude of $C$, but preserves the blackboard longitude (and exhibits the equivalence of all $\E$-transverse knots $C$ in $\R^4$ as mentioned in (a).) When $K$ is an unknot, the mod 2 residue of $\Delta_\nu$ is dual to a $(1,1)$-curve relative to the 0-longitude of $C$, as required by Example~\ref{R4mod2}, since transverse knots in $\R^3$ have odd self-linking.

\item[d)] Specializing further to an unknotted torus $\Sigma=S^1\times S^1\subset\R^2\times\R^2=M_r$ reveals other novel phenomena. There is an isotopy $F$ sending $\Sigma$ to a transverse torus $\Sigma_K$ as in (c) with $K$ and $C$ given by transverse unknots. We choose $C$ to be given by a front projection with just one crossing. Then $l(C)=-1$, so $\lambda$ on $\Sigma_K$ follows the $(-1)$-framing of $C$ (relative to the 0-longitude). Thus, there is a simple choice of $F$ for which $\lambda$ corresponds to the difference of the obvious basis elements, the two $S^1$-factors of $\Sigma$. By Montesinos \cite[Theorem~5.4]{M}, every element of $\SL(2,\Z)$ fixing this difference mod 2 is realized by a self-isotopy of $\Sigma$. (For example, if we write $\Sigma$ as the boundary of $S^1\times D^2$, an even Dehn twist of the latter is isotopic to the identity in $\R^4$ since $\pi_1(\SO(3))\cong\Z_2$.) Since $\SL(2,\Z)$ acts transitively on primitive classes, the self-isotopies act transitively on primitive classes with odd coefficients. Equivalently, they act transitively on primitive classes in $H^1(\Sigma)$ with odd coefficients in the Poincar\' e (or algebraic) dual basis (i.e.~primitive classes reducing mod 2 to $\ds$ from Section~\ref{ds}). Thus, we can realize any primitive, odd class as $f_1^*\PD(\lambda)$ by preceding $F$ by a self-isotopy. But $l(K)$ can range over all negative odd integers, so by (a), every positive odd multiple of $PD(\lambda)$ can be realized as $\Delta_\nu(\Sigma_K)$. Hence, every class with odd coefficients can be realized as $f_1^*\Delta_\nu(\Sigma_K)$ for some $K$ and $F$. Equivalently, every even class is realized as $D_\nu(\hat F)$ for isotopies $\hat F$ from one such unknotted transverse torus to the others. Since an arbitrary Engel manifold $M$ contains every bounded region of some $M_r$ as above (Observation~\ref{tight}), and any unknotted torus in $M$ is (by definition) isotopic to the resulting embedded $\Sigma$, we have:

\begin{prop}\label{unknottedTori}
Every unknotted torus $\Sigma'$ in an Engel manifold $M$ is isotopic to transverse tori with $\Delta_T=0$ but $\Delta_\nu$ realizing every class in $H^2(\Sigma')$ that reduces to $\ds$ mod 2. \qed
\end{prop}

\noindent Alternatively, suppose $\Sigma''$ is an unknotted transverse torus in $M$ with $\Delta_T=0$, and $\Delta\in H^2(\Sigma'')$ is any class reducing to $\ds$ mod 2 and with the same divisibility as $\Delta_\nu(\Sigma'')$. Then the previous reasoning gives an isotopy of $\Sigma''$ to itself, via the above $\Sigma$, with $f_1^*\Delta_\nu(\Sigma'')=\Delta$. Since the constant isotopy of $\Sigma''$ is clearly transverse, we conclude:
\begin{prop} For unknotted transverse tori with $\Delta_T=0$ in any Engel manifold, the only transverse isotopy invariant that is absolute (independent of choice of isotopy) and can be extracted from $\Delta_\nu$ is its divisibility. \qed
\end{prop}
\noindent We return to unknotted tori in Remark~\ref{overtwistedR3}.
\end{example}

\subsection{Varying $\Delta_T$}\label{DT}

We now present a method for varying $\Delta_T$. This is especially useful in prolongations, where we exhibit isotopic families of transverse tori realizing all classes in $H^1(\Sigma)$ as $\Delta_T$ (Example~\ref{T3}). When $\Delta_\nu$ is well-defined in $H^1(\Sigma)$, we find families realizing all linearly dependent pairs with even sum as $(\Delta_T,\Delta_\nu+\delta)$ for some fixed class $\delta$ (Example~\ref{vert}(c)). Families without such a relation seem harder to construct (cf.~Question~\ref{ques}(b)), although Section~\ref{Overtwisted} realizes all pairs with the correct mod 2 residue in overtwisted Engel manifolds by varying the Engel structure. Our main lemma is the following:

\begin{lem}\label{dt}
Suppose $\Sigma$ is an $\E$-transverse Legendrian torus (Definition~\ref{Leg}) in an Engel manifold $M$, and its characteristic foliation has a nondegenerate closed leaf $L$ along which $\Sigma$ has no $\W$-tangencies. Then there are $C^0$-small isotopies from $\Sigma$ to a family of transverse tori realizing all multiples of $\PD[L]$ as $\Delta_T$. For each such $\Delta_T$, there is a $C^0$-small transversely homotopic subfamily realizing all multiples of $\PD[L]$ congruent mod 2 to $D_T(F)$ as $D_\nu(F)$,  where $F$ is the induced isotopy from the transverse pushoff $\tau\Sigma$.
\end{lem}

\noindent These hypotheses are more restrictive than those of Lemma~\ref{dnu}. For example, the Legendrian tori of Theorem~\ref{main1} (made from arbitrary tori to prove Theorem~\ref{main0} by transverse pushoff) are constructed with circles of $\W$-tangencies that intersect each leaf. The transverse tori in Example~\ref{PV}(a) can be realized as transverse pushoffs of Legendrian tori by taking the profile $K$ to be the pushoff of a Legendrian knot, but the cusps of the latter generate $\W$-tangencies intersecting the leaves. Nevertheless, the lemma suffices to prove Theorem~\ref{nonhom} and generate other interesting examples.

\begin{proof}
First, we isotope $\Sigma$ to a new $\E$-transverse Legendrian torus  $\Sigma'$ with Reeb components inserted into its characteristic foliation parallel to $L$: By repeatedly applying Scholium~\ref{Reeb} to $\Sigma$, we can split $L$ into a collection of parallel closed leaves separated by any even number of Reeb components, all convex in the same direction (as in Figure~\ref{ReebFig}). To realize convexity in the opposite direction, precede this construction by folding along $L$ (cf.\ proof of Scholium~\ref{Reeb}). This splits $L$ into three closed leaves but only changes the characteristic line field by a small perturbation.  Thus, the torus remains Legendrian if we suitably perturb it in the $\W$-direction. Now the new central leaf attracts in the opposite direction from $L$ (with time reversed), and the previous construction gives Reeb components with the opposite convexity. We can assume the isotopy preserves $L$ and fixes a neighborhood of some circle $C^\parallel$ parallel to $L$.

The transverse pushoff $\tau\Sigma'$ (Proposition~\ref{push}) inherits the nondegenerate closed leaf $L$ and the new Reeb components from $\Sigma'$, and it agrees with $\tau\Sigma$ near $C^\parallel$. As in Proposition~\ref{comp}, $\Delta_T(\tau\Sigma')$ must be a multiple of $\PD[L]$, and any coefficient can be realized by controlling the Reeb components. Since $\tau\Sigma'$ agrees with $\tau\Sigma$ near $C^\parallel$, $D_\nu(F)$ vanishes on $[L]$, so is a multiple of $\PD[L]$. By Corollary~\ref{Z}, this can be any multiple with $D_T(F)+D_\nu(F)$ even, and the resulting transverse tori are $C^0$-small transversely homotopic.
\end{proof}

\begin{cor}
Let $\Sigma$ be a torus with a tight neighborhood in a contact 3-manifold $N$. Then $\Sigma$ has a lift to the prolongation $\pr N$ that is isotopic to transverse tori representing infinitely many transverse homotopy classes.
\end{cor}

\begin{proof}
We can assume (for example via convexity and the Flexibility Theorem~\ref{flex}) that the characteristic foliation on $\Sigma$ is nonsingular and has a nondegenerate closed leaf. Then the above lemma applies to the Legendrian lift of $\Sigma$ from Example~\ref{liftReeb}. (That example produces infinitely many homotopy classes of such lifts if we vary Reeb components before applying the lemma; cf.\ Remark~\ref{immersion}.) The resulting transverse homotopy classes are distinguished by the multiplicity of $\Delta_T$.
\end{proof}

\begin{example}\label{T3}
The standard tight contact 3-torus $T^3$ is given by $(\R^3/\Z^3,\sin(2\pi x)dy+\cos(2\pi x)dz)$. Let $\Sigma_c\subset T^3$ be the 2-torus obtained by setting $x$ equal to a constant $c$. Then the characteristic line field on $\Sigma_c$ is constant, but realizes all slopes as we vary $c$ (an isotopy). Perturbing $\Sigma_c$ shows that $\Sigma_0$ is isotopic to a convex torus with a closed leaf realizing any preassigned primitive class in $H_1(\Sigma_0)$. Adding Reeb components and taking Legendrian lifts to $\pr T^3=T^4$ realizes (up to isotopy) all lifts of $\Sigma_0$ (which are classified by $H^1(\Sigma_0)=\Z\oplus\Z$). The conclusion of the corollary applies to each of these. For the lift without Reeb components, Lemma~\ref{dt} realizes all elements of $H^1(\Sigma_0)$ as $\Delta_T$. This class (not just its multiplicity) is an invariant of transverse homotopy since every self-homotopy of a lift in $T^4$ is the identity on $H^1$. One can make similar examples where this last statement fails by replacing $T^3$ by other $T^2$-bundles over $S^1$. More generally, one can vary $\Delta_T$ in this manner whenever we have a half-unit of {\em Giroux torsion}, namely an embedding of the subset of $T^3$ given above by the restriction $0\le x\le 1/2$. (Many such examples are known.)
\end{example}

We can sometimes control transverse homotopy classes and transverse isotopy classes simultaneously. We implicitly restrict the former to a single isotopy class. (Families that are transversely homotopic but not isotopic arise in Example~\ref{PV}.)

\begin{thm}\label{prolong}
Let $\Sigma=\partial N_0\subset N$ be a torus bounding a compact submanifold $N_0$ of a 3-manifold, and let $\xi^*$ be a plane field on $N$ with $e(\xi^*)|N_0=0$. Then there is a contact structure $\xi$ homotopic to $\xi^*$ and a lift of $\Sigma$ to the prolongation $\pr(N,\xi)$ whose isotopy class contains an infinite collection of transverse homotopy classes of transverse tori, for which each class contains infinitely many transverse isotopy classes.
\end{thm}

\begin{proof}
First, we suitably define $\xi$ and the corresponding lift of $\Sigma$. Since $e(\xi^*)|N_0=0$, we can choose an oriented line field $\L^*$ in $\xi^*|N_0$. The resulting flag $\L^*\subset\xi^*|N_0\subset TN_0$ trivializes $TN_0$. To arrange a nonsingular characteristic foliation on $\Sigma$, let $v$ be a vector field in $T\Sigma$ determining its canonical framing. Since $\chi(N_0)=0$, $v$ has a nonvanishing extension to $TN_0$. Using the flag trivialization, we can interpret the extension as a map $N_0\to S^2$, whose restriction to $\Sigma$ must be homologically, hence homotopically, trivial. Thus, near $\Sigma$, we can homotope $\xi^*$ to be transverse to $v$. Then its characteristic line field on $\Sigma$ is nonsingular and determines its canonical framing. The angle of $\L^*$ in $\xi^*|\Sigma$ relative to $T\Sigma$ now gives a map $\Sigma\to S^1$. We can assume this vanishes near some essential circle $C$ in $\Sigma$ where $\L^*$ agrees with the oriented characteristic line field. After further homotopy of $\xi^*$, we can also assume its characteristic foliation (which determines the canonical framing on $T\Sigma$) contains $C$ as a leaf and agrees with some convex local model on $\Sigma$, and that $\xi^*$ agrees with the corresponding contact structure near $\Sigma$. By Eliashberg~\cite{E}, we can homotope $\xi^*$ elsewhere to obtain a (typically overtwisted) contact structure $\xi$ on $N$. These homotopies send $\L^*\subset\xi^*|N_0$ to a line field $\L$ in $\xi|N_0$. This $\L$ lifts $N_0$ to $\hat N_0\subset\pr(N,\xi)$. Then $\partial \hat N_0$ agrees with the Legendrian lift $\hat\Sigma$ of $\Sigma$ near $C$, but differs from it elsewhere. However, by isotoping $\Sigma$ in $N$ rel $C$, we can add Reeb components to its foliation as in Example~\ref{liftReeb} along a  closed leaf parallel to $C$, changing the homotopy class of the Legendrian lift $\hat\Sigma$ rel $C$ so that it becomes vertically isotopic to $\partial\hat N_0$. After isotoping $\hat N_0$, we can now assume the Legendrian lift $\hat\Sigma$ equals $\partial\hat N_0$. Equivalently, we homotope $\L$ on $N_0$ so that its restriction to $\Sigma$ agrees with our new characteristic line field.

To apply $\Delta_\nu$, we must see that $\PD[C]$ in $H^1(\Sigma)$ is nonzero on the annihilator $A$ of $H_3(\pr(N,\xi))$. Let $C_0\subset \Sigma=\partial N_0$ be an essential circle bounding a surface $\Sigma_0\subset N_0$, so $[C_0]\in A$. Let $\tau_\L$ and $\tau_\Sigma$ be the framings of $TN|C_0$ determined, respectively, by  the flag $\L|C_0\subset\xi|C_0\subset TN|C_0$ and by the tangents to $C_0\subset \Sigma\subset N$. Then $w_2(TN|\Sigma_0,\tau_\L)=0$ since the framing extends. But the normal to $C_0$ in $\Sigma$ extends normally over $\Sigma_0$, so $w_2(TN|\Sigma_0,\tau_\Sigma)=w_2(T\Sigma_0,TC_0)=\chi(\Sigma_0)|2\ne0$. Thus, the two framings $\tau_\L$ and $\tau_\Sigma$ differ by the nontrivial element of $\pi_1(\SO(3))$. Since $\L$ coincides with the characteristic line field of $\Sigma$, this means the latter has odd winding number as we traverse $C_0$. Equivalently, $\rho C\cdot C_0$ is odd, where $2\rho$ is the signed count of Reeb components of the foliation, so $\PD[C]$ has nonzero value on $[C_0]$.

To complete the proof, we apply Lemma~\ref{dt} to $\hat\Sigma$ with $L=C$. This realizes all multiples of $\PD[C]$ as $\Delta_T$ by $C^0$-small isotopies, and the multiplicities distinguish infinitely many transverse homotopy classes. Within each such class, the induced isotopies $F$ from $\tau\hat\Sigma$ realize all multiples of $\PD[C]$ congruent to $D_T(F)$ mod 2 as $D_\nu(F)$. The normal framings of $\tau\hat\Sigma$ induced by the Seifert solid $\hat N_0$ and $\W$ agree since $\hat N_0$ lifts $N_0$. Thus, $\Delta_\nu(\tau\hat\Sigma)=0$, where we work relative to $\hat N_0$, so $\Delta_\nu$ of each new transverse torus is a multiple of $\PD[C]$. These classes range over all multiples for which $\Delta_T+\Delta_\nu$ is odd. (Note that the sum has mod 2 reduction $\ds$ that cannot be 0, since it is isotopy invariant with $\Delta_\nu(\tau\hat\Sigma)=0$ and $\langle\Delta_T(\tau\hat\Sigma),C_0\rangle=\rho C\cdot C_0$ odd.) By Corollary~\ref{wellDef}, $\Delta_\nu$ is well-defined on $A$, independently of the choice of Seifert solid. Since $\PD[C]$ has nonzero value on $[C_0]\in A$, the resulting multiplicities distinguish infinitely many transverse isotopy classes within each transverse homotopy class.
\end{proof}

\begin{proof}[Proof of Theorem~\ref{nonhom}]
We are given a circle bundle over a 3-manifold $N$, lying in a plane bundle $\xi^*$ with $e(\xi^*)$ even. Then $\xi^*\oplus\R$ is trivial, so we can identify it with $TN$. The above theorem immediately applies to many choices of $\Sigma$. If $\Sigma$ is unknotted in an $\R^3\subset N$, then it bounds a solid torus $N_0\subset\R^3$ and $e(\xi^*)|N_0=0$ trivially. We can choose $\L^*$ to extend over the trivial bundle $\xi^*|\R^3$. Then all of $\R^3$ lifts by $\L$, so $\hat\Sigma$ is unknotted in $\pr(N,\xi)$. (The other choices of $\L^*$ realize all homotopy classes of lifts of $N_0$ by families as above, and we can just as easily lift solid tori following nontrivial loops in $N$.)
\end{proof}

\begin{Remark}\label{overtwistedR3}
By Proposition~\ref{unknottedTori} for $\Delta_T=0$ and its same method in general, we conclude that an unknotted torus in an Engel manifold as constructed above is isotopic to transverse tori realizing $(\Delta_T,\Delta_\nu)$ by all linearly dependent pairs with sum reducing mod 2 to $\ds$. This doesn't seem to extend to arbitrary Engel manifolds as in Proposition~\ref{unknottedTori}, since at minimum, we are using a large $w$-interval in the prolongation of an overtwisted contact structure on $\R^3$.
\end{Remark}

We again finish with explicit computations:

\begin{example}\label{vert}
a) Let $K$ be a transverse knot in a contact 3-manifold $(N,\xi)$, and let $\Sigma\subset\pr N$ be the torus obtained by restricting the circle bundle to $K$. This is $\E$-transverse Legendrian, although not generic since its characteristic line field is $\W|\Sigma$. The topological knot type of $\Sigma$ depends on that of $K$. For example, when $N$ is simply connected and $e(\xi)=0$, the homotopy type of $N-K$ (hence the knot group) is preserved via the universal cover of $(\pr N,\Sigma)$. The variable line field $\L|K$ in $\xi|K=\nu K$ rotates once (negatively) as we follow the characteristic foliation (by fibers of $\pr N$) around $\Sigma$ once. By the proof of Proposition~\ref{push}, we obtain the transverse pushoff $\tau\Sigma$ by pushing $K$ in the positive $\L$-direction, so $\tau\Sigma$ projects to the boundary of a tubular neighborhood $N_0$ of $K$ in $N$, with $\L$ directed outward. Thus, we can equivalently view $\tau\Sigma$ as the transverse lift of $\partial N_0$. (We have chosen to exhibit the boundary orientation on $\tau\Sigma$; the other orientation would correspond to a push in the negative $\L$ direction.) The transverse lift of $\partial N_0$ is not a lift as occurs in the previous proof since it does not extend over $N_0$. However, we can still analyze $\Delta_T$. We can assume $\partial N_0$ is foliated by helices (proof of Proposition~\ref{helices}), so $$\Delta_T(\tau\Sigma)=0$$ (Proposition~\ref{comp}(b)). After a perturbation of $\partial N_0$, we can assume there is a nondegenerate closed leaf $L$ representing any given primitive class $$[L]=p\mu-q\lambda$$ with $p,q\in\Z^+$ and $p/q$ sufficiently large relative to a given choice of longitude $\lambda$. (See Section~\ref{OrContSurf} for the sign of $L$.) Thus, applying Lemma~\ref{dt} to the Legendrian lift of $\partial N_0$ gives infinitely many transverse homotopy classes isotopic to $\tau\Sigma$, realizing all multiples of $\PD[L]$ as $\Delta_T$ for each such $L$.

\item[b)] Now suppose that $K$ is nullhomologous in $N$. The circle bundle $\pr N$ restricted to any Seifert surface of $K$ is a Seifert solid for $\Sigma$. This becomes a Seifert solid for $\tau\Sigma$ under the isotopy producing the latter from $\Sigma$. The invariant $\Delta_\nu(\tau\Sigma)$ compares the Engel framing (given by $\L$, which is outward normal to $\partial N_0$ when measuring its lift) against the outward normal to the boundary of the Seifert solid. These framings agree on the longitude $\lambda_0$ determined by the 0-framing of $K$. Thus, $\langle \Delta_\nu(\tau\Sigma),\lambda_0\rangle=0$. But around the meridian $\mu$, $\L$ rotates once positively, so $\langle \Delta_\nu(\tau\Sigma),\mu\rangle=1$. Since $\lambda_0\cdot\mu=-1$, we conclude that
$$\Delta_\nu(\tau\Sigma)=-\PD(\lambda_0).$$ Note that this only depends on the topological knot type of $K$. In particular, it is preserved under stabilization of $K$.  By Section~\ref{DeltaNu}, $\Delta_\nu$ is well-defined in $H^1(\tau\Sigma)$ if, for example, $N$ is open (so $H_3(N)=0$). It is well-defined in general if we restrict to Seifert solids made from the fibration as above.

Applying Lemma~\ref{dnu} to circles transverse to the leaves $L$ specified in (a) realizes infinitely many transversely homotopic tori, with $\Delta_T=0$ and $\Delta_\nu$ realizing all classes of the form 
$$\Delta_\nu=2m\PD(\mu)+(2n-1)\PD(\lambda_0)$$
with $n>\min(-\epsilon m,0)$ for some $\epsilon>0$ (cf.~Example~\ref{PV}(b)). For examples with $\Delta_T\ne0$, fix $L$ as in (a). Then Lemma~\ref{dt} gives transverse tori realizing
$$\Delta_T=m\PD[L] {\rm\ \ and\ \ } \Delta_\nu=n\PD[L]-\PD(\lambda_0)$$
for all $m,n\in\Z$ with $m+n$ even. These are transversely homotopic for each fixed value of $\Delta_T$.

\item[c)] Now we modify $\xi$ by a {\em Lutz twist} along $K$. This basically inserts a unit of Giroux torsion along $\partial N_0$, making $\xi$ overtwisted. Then we can realize all primitive classes as $[L]$ by isotopy of $\Sigma$. Thus, in (a) we can realize all classes in $H^1(\Sigma)$ as $\Delta_T$ (and realize transverse tori in all other homotopy classes of lifts) as in Example~\ref{T3}. In (b), we can realize all linearly dependent pairs in $H^1(\Sigma)$ with even sum as $$(\Delta_T,\Delta_\nu+\PD(\lambda_0)).$$
\end{example}


\begin{thebibliography}{MM}

\bibitem[BEM]{BEM}
M.\ Borman, Y.\ Eliashberg and E.\ Murphy,
{\em Existence and classification of overtwisted contact structures in all dimensions}, 
Acta Math.\ {\bf 215} (2015), no.\ 2, 281--361.

\bibitem[C]{C}
E.\ Cartan,
{\em Sur quelques quadratures dont l'\'el\'ement diff\'erentiel contient des fonctions arbitraires}, 
Bull.\ Soc.\ Math.\ France {\bf 29} (1901), 118--130.

\bibitem[CPP]{CPP}
R.\ Casals, \'A.\ del Pino and F.\ Presas,
{\em Loose Engel structures}, 
Compositio Math.\ {\bf 156} (2020), 412--434.

\bibitem[E]{E}
Y.\ Eliashberg,
{\em Classification of overtwisted contact structures
on 3-manifolds}, 
Invent.\ Math.\ {\bf 98} (1989), no.\ 3, 623--637.

\bibitem[EM]{EM}
Y.\ Eliashberg and N.\ Mishachev,
{\em Introduction to the h-Principle},
Grad.\ Studies in Math.\ {\bf 48},
Amer.\ Math.\ Soc., Providence, 2002. 

\bibitem[FK]{FK}
M.\ Freedman and R.\ Kirby,
{\em A geometric proof of Rochlin's Theorem}, 
Proc.~Symp.~Pure Math.~{\bf 32} (1978), 85--97.

\bibitem[Gi]{G}
E.\ Giroux,
{\em Convexit\' e en topologie de contact}, 
Comm.~Math.~Helv.~{\bf 66} (1991), 637--677.

\bibitem[Go]{Ann}
R.\ Gompf,
{\em Handlebody construction of Stein surfaces},
Ann.\ Math.\ (2) {\bf148} (1998), 619--693.

\bibitem[Gray]{Gray}
J.\ W.\ Gray,
{\em Some global properties of contact structures},
Ann.\ Math.\ (2) {\bf 69} (1959), 421--450.

\bibitem[Gro]{Gr}
M.\ Gromov,
{\em Partial Differential Relations},
Ergebnisse der Mathematik und ihrer Grenzgebiete (3), Springer--Verlag,
Berlin, 1986.

\bibitem[H]{H}
K.\ Honda,
{\em On the classification of tight contact structures I}, 
Geom.\ Topol.\ {\bf 4} (2000), 309--368.

\bibitem[K]{K}
M.~Kegel,
{\em Non-isotopic transverse tori in Engel manifolds},
Rev.\ Mat.\ Iberoam., {\bf 40} (2024), no.~1, 43--56.

\bibitem[KM]{KM}
P.\ Kronheimer and T.\ Mrowka,
{\em Witten's conjecture and property P},
Geom.\ Topol.\ {\bf 8} (2004) 295--310.

\bibitem[McD]{McD}
D.\ McDuff,
{\em Applications of convex integration to symplectic and contact geometry},
Ann.\ Inst.\ Fourier (Grenoble) {\bf 37}
(1987), no.\ 1, 107--133.

\bibitem[M]{M}
J.\ Montesinos,
{\em On twins in the 4-sphere I},
Quart.\ J.\ Math.\ Oxford, Ser.\ (2) {\bf 34} (1983) No.\ 134, 171--199.

\bibitem[OS]{OS}
B.\ Ozbacgi and A.\ Stipsicz,
{\em Surgery on Contact 3-Manifolds and Stein Surfaces},
Bolyai Soc.\ Math.\ Studies {\bf 13}, 
Springer--Verlag Berlin, Heidelberg, 2004.

\bibitem[PP]{PP}
\'A.\ del Pino and F.\ Presas,
{\em Flexibility for tangent and transverse immersions in Engel manifolds}, 
Rev.~Mat.~Comp.\ {\bf 32(1)} (2019), 215--238.

\bibitem[PV]{PV}
 Á.~del Pino and T.~Vogel
{\em The Engel--Lutz twist and overtwisted Engel structures},
Geom.\ Topol.\ {\bf 24} (2020) 2471--2546.

\bibitem[Po]{Po}
L.\ Pontryagin,
{\em A classification of mappings of the three-dimensional complex into the two dimensional sphere},
Matematicheskii Sbornik {\bf 9(51)} (1941),  331--363.

\bibitem[Pr]{P}
F.~Presas, 
{\em Non-integrable distributions and the h-principle},
Eur.\ Math.\ Soc.\ Newsl.\ {\bf 99} (2016), 18--26.

\bibitem[R]{R}
V.~Rokhlin, 
{\em Proof of Gudkov's hypothesis},
Functional Anal.~Appl.~{\bf 6} (1972), 136--138.

\end{thebibliography}
\end{document}